\documentclass[twocolumn,english]{IEEEtran}
\usepackage[T1]{fontenc}
\usepackage{babel}
\usepackage{calc}
\usepackage{textcomp}
\usepackage{amsmath}
\usepackage{amsthm}
\usepackage{amssymb}
\usepackage{graphicx}
\usepackage[numbers]{natbib}
\usepackage{bm}
\usepackage[unicode=true,
 bookmarks=true,bookmarksnumbered=true,bookmarksopen=true,bookmarksopenlevel=1,
 breaklinks=false,pdfborder={0 0 0},pdfborderstyle={},backref=false,colorlinks=false]
 {hyperref}
\hypersetup{pdftitle={Your Title},
 pdfauthor={Your Name},
 pdfpagelayout=OneColumn, pdfnewwindow=true, pdfstartview=XYZ, plainpages=false}

\makeatletter

\newcommand{\lyxmathsym}[1]{\ifmmode\begingroup\def\b@ld{bold}
  \text{\ifx\math@version\b@ld\bfseries\fi#1}\endgroup\else#1\fi}

\theoremstyle{plain}
\newtheorem{thm}{\protect\theoremname}
\theoremstyle{plain}
\newtheorem{lem}{\protect\lemmaname}
\theoremstyle{remark}
\newtheorem{rem}{\protect\remarkname}
\theoremstyle{definition}
\newtheorem{example}{\protect\examplename}

\usepackage{algorithmic}
\usepackage[bb=boondox]{mathalfa}
\makeatletter
\def\thmhead@plain#1#2#3{%
  \thmname{#1}\thmnumber{\@ifnotempty{#1}{ }\@upn{#2}}%
  \thmnote{ {\the\thm@notefont#3}}}
\let\thmhead\thmhead@plain
\makeatother
\usepackage{xcolor}

\usepackage{subfig}
\usepackage{babel}
\makeatletter

\usepackage{array}
\usepackage{boldline}
\usepackage{multirow}
\usepackage{multicol}

\makeatother

\providecommand{\examplename}{Example}
\providecommand{\lemmaname}{Lemma}
\providecommand{\remarkname}{Remark}
\providecommand{\theoremname}{Theorem}

\newdimen\opdrop \opdrop=0.10ex
\makeatletter
\def\@opbox#1#2{\lower\opdrop\hbox{$\m@th#1\mathsf{#2}$}}
\def\T{{\mkern-2mu\mathchoice{\@opbox\displaystyle T}{\@opbox\textstyle T}%
                             {\@opbox\scriptstyle T}{\@opbox\scriptscriptstyle T}}}
\def\H{{\mkern-1mu\mathchoice{\@opbox\displaystyle H}{\@opbox\textstyle H}%
                             {\@opbox\scriptstyle H}{\@opbox\scriptscriptstyle H}}}
\makeatother

\begin{document}
\title{Worst-Case and Covariance-Reconstruction Robust Adaptive Beamforming: Equivalence and Closed-Form Solution}
\author{Licheng~Zhao, Rui~Zhou, Wenqiang~Pu, and Daniel~P.~Palomar,~\IEEEmembership{Fellow,~IEEE}\thanks{Licheng Zhao and Wenqiang Pu are with Shenzhen Research Institute of Big Data, China (email: \{zhaolicheng, wpu\}@sribd.cn).
Rui Zhou is with City University of Hong Kong (Dongguan), China (email: rui.zhou@cityu-dg.edu.cn). Daniel P. Palomar is with the Hong Kong
University of Science and Technology, Hong Kong (email: palomar@ust.hk).}}
\maketitle
\begin{abstract}
Robust adaptive beamforming (RAB) has developed along two separate
lines. The worst-case-constrained approach guards against steering vector
mismatch with a bounded uncertainty set, but suffers from desired signal
cancellation at high input signal-to-noise ratio (SNR). The
covariance-reconstruction approach instead estimates the interference-plus-noise
covariance matrix and solves for the mismatch numerically. We prove that the two
are related by primal-dual equivalence: once the worst-case beamformer is
normalized, it coincides exactly with the reconstruction-based one, so both
reduce to one formulation. The signal
cancellation of the worst-case approach is a consequence of its parameter
choice, not of the formulation. For the unified problem we
develop a closed-form solution in three stages: Diagonalization Transform,
Phase Alignment, and Karush-Kuhn-Tucker (KKT) Solution. Instead of solving
for the Lagrange multiplier, we introduce an auxiliary
variable that reduces the KKT system to a single scalar equation solvable by
bisection. For an $N$-sensor array the cost is $\mathcal{O}(N^{3})$, the same
order as the Lagrange-equation-based RMVB and well below the
$\mathcal{O}(N^{3.5})$ of a second-order cone program, and it covers the
rank-deficient case that RMVB excludes. We establish necessary and sufficient
conditions for existence and uniqueness, including a case where the feasible set
is nonempty yet no optimal solution exists, and show that the Lagrange
multiplier fixes the optimal value. Simulations confirm that, with matching
parameters, the worst-case beamformer attains the output
signal-to-interference-plus-noise ratio (SINR) of the reconstruction-based one,
and that the closed-form solution reproduces the optimum of convex solvers in a
fraction of their time.
\end{abstract}


\section{Introduction\label{sec:Introduction-new}}

\IEEEPARstart{R}{obust} adaptive beamforming (RAB) is among the oldest continuously
active problems in array signal processing \citep{van2002optimum}, and the reason
it has stayed active is that its central difficulty has never gone away. An adaptive beamformer is
designed from an assumed model of the array response to the signal of interest.
That model is never exactly correct. Calibration drifts, element coupling, and the
propagation medium distort the wavefront, the direction of arrival (DOA) is estimated
rather than known, and the array itself may flex. The classical minimum variance
distortionless response (MVDR) beamformer is acutely sensitive to these discrepancies,
and the entire literature that follows is a response to that sensitivity.

The first remedy was regularization of the sample covariance matrix. It was derived and
analyzed in the Soviet literature by Abramovich \citep{abramovich1981controlled}
and Cheremisin \citep{cheremisin1982efficiency} in the early 1980s, and developed
independently in the West by Cox, Zeskind, and Owen \citep{cox1987robust} and by
Carlson \citep{carlson1988covariance}, to whom the concept of diagonal loading is
usually credited. The technique worked, and it is still working: what it lacked,
and lacked for two decades, was a principled account of how much loading to
apply. Cox and coauthors offered the empirical rule of twice the noise power, and
practice largely followed suit.

The modern era commenced when robustness was recast as optimization. Vorobyov,
Gershman, and Luo \citep{vorobyov2003robust} modeled the mismatch as a vector
confined to a bounded uncertainty region and required the response to be
maintained for every mismatch in that region, producing a convex program with a
guarantee attached. Li, Stoica, and Wang \citep{li2003robust} arrived at a
closely related formulation from the covariance-fitting direction and showed that
its solution is a diagonally loaded beamformer whose loading level is determined,
rather than chosen. Lorenz and Boyd \citep{lorenz2005robust} then gave an
efficient solution method for the worst-case problem. Within a few years the
field had converted a heuristic into a design principle.

A second line of attack emerged a decade later. Rather than protecting the
beamformer against an uncertain steering vector, Gu and Leshem
\citep{gu2012robust} proposed reconstructing the interference-plus-noise
covariance matrix directly from the snapshot data, excluding the angular sector
containing the desired signal, and correcting the steering vector afterwards.
This largely removes the desired signal cancellation defect that degrades the
worst-case approach at high input signal-to-noise ratio (SNR), and it prompted a
substantial body of work refining the reconstruction
\citep{huang2015robust,huang2019new,zhu2020robust,luo2023urglq}. The two lines
have since been treated as alternatives, each with acknowledged strengths, and
the choice between them has been made empirically.

The field has not settled. Vorobyov and coauthors have continued to develop the
worst-case formulation, most recently over distributional and Wasserstein
uncertainty sets \citep{hassasirani2025sinr,hassasirani2025wasserstein}, and a
recent unification of several robustness paradigms in this journal
\citep{wang2025distributionally} shows both how much structure the problem still
conceals and how actively that structure is being sought. As recently as this
year, Huang, Huang, Vorobyov, and Luo \citep{huang2026optimal} established a new
equivalence between the maximin and minimax formulations under weaker regularity
conditions than the result it supersedes, which is two decades old. Diagonal
loading, meanwhile, has not been displaced; it keeps reappearing as the special
case that each new formalism recovers. A problem whose original authors are still
proving new structural theorems about it, four decades on, is not a closed one.

The setting has also changed. Arrays now run to hundreds of elements in massive
multiple-input multiple-output (MIMO) systems and radio astronomy
\citep{peng2024real} and to a thousand in wideband radar
\citep{delafrooz2026tiled}; the weights are recomputed continuously, and often on
the platform itself, where a general-purpose convex solver is rarely an option.
The field is accordingly pursuing lower-complexity robust beamformers
\citep{mohammadzadeh2023efficient,herath2025scalable,yan2025analytically}. A
closed-form solution answers that demand directly. It places robust beamforming
at the same $\mathcal{O}\left(N^{3}\right)$ as the non-robust MVDR beamformer it
replaces, using a factorization, an eigendecomposition and a scalar bisection. Because
the solution is explicit, it makes it possible to say when a solution
exists and when it is unique, which no solver output can.

The demand for robustness has broadened in step. Mismatch is named as a limiting
factor in integrated sensing and communication (ISAC), where mutual coupling between
elements distorts the array response \citep{maeng2026robust}; in near-field
extremely large-scale arrays, where beam focusing depends on user localization
that is inevitably imperfect \citep{mozaffarikhosravi2026robust}; in low Earth
orbit satellite constellations, where direction-of-arrival error can consume the
greater part of the achievable throughput \citep{wang2026hybrid}; in towed sonar
arrays, which deform in tow and where correcting for it recovers on the order of
ten decibels of SNR in sea trials \citep{xing2025robust}; and
in medical ultrasound, where uncertainty in tissue sound speed defocuses the beam
unless it is designed against \citep{fan2025computational}. The same structure
appears outside array processing altogether. Minimum-variance portfolio
selection minimizes a quadratic objective subject to a normalized linear budget, differs
from the beamforming problem only in the interpretation of its ingredients, and
confronts the identical difficulty that the covariance and the direction of
interest are estimated rather than known
\citep{palomar2025portfolio}.
Robust adaptive beamforming is, in this sense, a canonical instance of a problem
that recurs wherever a quadratic function is minimized under estimation uncertainty.

Against this background, the present paper makes three contributions. 
\begin{itemize}
\item First, we show that the worst-case-constrained and reconstruction-based approaches are not alternatives at all: they are related by a primal-dual equivalence and reduce to a common formulation under different choices of parameters. As a result, the desired signal cancellation long regarded as an intrinsic defect of worst-case beamforming is instead a consequence of how its parameters are set, and with an appropriate choice the worst-case beamformer performs comparably to
the reconstruction-based one.
\item Second, we solve that common formulation in closed
form, through three successive stages that reduce the optimality conditions to a
single scalar equation, with lower complexity than
general-purpose solvers and without their overhead.
\item Third, because the solution
is available in closed form, we can characterize when it exists and when
it is unique. Both questions turn out to have non-obvious answers: an optimal
solution may fail to exist even when the problem is feasible, and uniqueness can
survive the loss of strict convexity that accompanies a rank-deficient sample
covariance. These conditions specify the parameter choices for which the problem is
well posed, a question that cannot be settled by inspecting a solver's output. 
\end{itemize}  

The remainder of the paper is organized as follows. Section
\ref{sec:Signal-Model-Existing} sets out the signal model and reviews the two
mainstream families of robust formulations together with the methods used to
solve them. Section \ref{sec:Connection-data-model} establishes the primal-dual
equivalence between them. Section \ref{sec:Problem and Solution} derives the
closed-form solution, and Section \ref{sec:Existence-and-Uniqueness} explains
the existence and uniqueness conditions, with illustrative examples. Section
\ref{sec:Numerical-Simulations} reports numerical results and Section
\ref{sec:Conclusion} concludes.

The following notation is adopted. Boldface upper-case letters represent
matrices, boldface lower-case letters denote column vectors, and standard
lower-case or upper-case letters stand for scalars. $\mathbb{R}$
($\mathbb{C}$) denotes the real (complex) field. $\exp\left(\cdot\right)$
stands for the natural exponential function. $\left|\cdot\right|$
denotes the absolute value for the real case and modulus for the complex
case. $\mathrm{Re}\left[\cdot\right]$ and $\mathrm{Im}\left[\cdot\right]$
denote the real and imaginary part of a complex number. $\arg$ represents
the phase angle of a complex number. $\mathbb{E}\left[\cdot\right]$
represents the expectation of a random variable. $\Pr\left\{ \cdot\right\} $
is the probability operator. $\left\Vert \cdot\right\Vert _{p}$ denotes
the $\ell_{p}$-norm of a vector. $\odot$ stands for the Hadamard
product. $\mathbf{X}^{\T}$ and $\mathbf{X}^{\H}$ denote the transpose
and conjugate transpose of $\mathbf{X}$. $\mathbf{X}\succeq\mathbf{0}$
means $\mathbf{X}$ is positive semidefinite. $\textrm{Diag}\left(\mathbf{x}\right)$
is a diagonal matrix with $\mathbf{x}$ filling its principal diagonal.
$\mathrm{BlkDiag}\left(\left\{ \mathbf{X}_{n}\right\} \right)$ is
a block-diagonal matrix placing $\left\{ \mathbf{X}_{n}\right\} $
sequentially on the main diagonal. The superscript $\star$ represents
the optimal solution of some optimization problem. Whenever arithmetic
operators ($>$, $\geq$, $\left|\cdot\right|$, $\cdot/\cdot$, etc.)
are applied to vectors or matrices, they are elementwise operations.

\section{Signal Model and Existing Approaches\label{sec:Signal-Model-Existing}}

Adaptive beamforming \citep{griffiths1982alternative,shan1985adaptive} is a
classical field with extensive application in wireless communications and radar
signal processing. Consider a uniform linear array (ULA)
of $N$ sensors receiving narrowband signals. The received snapshot vector $\mathbf{x}_{r}\left[t\right]\in\mathbb{C}^{N}$
at time $t$ ($t=1,\ldots,T$) is modeled as

\begin{equation}
\begin{aligned}\mathbf{x}_{r}\left[t\right] & =\mathbf{x}_{s}\left[t\right]+\mathbf{x}_{i}\left[t\right]+\mathbf{x}_{n}\left[t\right]\\
 & =\mathbf{a}\cdot s\left[t\right]+\mathbf{x}_{i}\left[t\right]+\mathbf{x}_{n}\left[t\right],
\end{aligned}
\end{equation}
where $\mathbf{x}_{s}\left[t\right]$, $\mathbf{x}_{i}\left[t\right]$, and
$\mathbf{x}_{n}\left[t\right]$ are the mutually independent desired signal,
interference, and noise components, $s\left[t\right]$ is the signal waveform,
$\mathbf{a}$ the steering vector associated with the signal DOA $\theta$, and $\mathbf{x}_{n}\left[t\right]$ white background noise. With
weight vector $\mathbf{w}\in\mathbb{C}^{N}$, the beamformer output is

\begin{equation}
y_{o}\left[t\right]=\mathbf{w}^{\H}\mathbf{x}_{r}\left[t\right].
\end{equation}

The prevalent design is the MVDR method \citep{capon1969high}, which maximizes the output
signal-to-interference-plus-noise ratio (SINR)

\begin{equation}
\mathrm{SINR}=\frac{\sigma_{s}^{2}\left|\mathbf{w}^{\H}\mathbf{a}\right|^{2}}{\mathbf{w}^{\H}\mathbf{R}\mathbf{w}},
\end{equation}
where $\sigma_{s}^{2}$ is the desired signal power and $\mathbf{R}$ is meant to be
the interference-plus-noise covariance matrix $\mathbf{R}_{i+n}$. Maximizing the
output SINR is equivalent to

\begin{equation}
\begin{aligned} & \underset{\mathbf{w}\in\mathbb{C}^{N}}{\mathsf{minimize}} &  & \mathbf{w}^{\H}\mathbf{R}\mathbf{w}\\
 & \mathsf{subject\ to} &  & \mathbf{w}^{\H}\mathbf{a}=1.
\end{aligned}
\label{eq:MVDR}
\end{equation}
In practice $\mathbf{R}_{i+n}$ is unknown and the sample covariance matrix
$\hat{\mathbf{R}}=\frac{1}{T}\sum_{t=1}^{\T}\mathbf{x}_{r}\left[t\right]\mathbf{x}_{r}^{\H}\left[t\right]$
is used in its place. When the snapshots are plentiful, $\mathbf{R}\succ\mathbf{0}$
typically holds and \eqref{eq:MVDR} admits the closed-form solution

\begin{equation}
\mathbf{w}^{\star}=\frac{\mathbf{R}^{-1}\mathbf{a}}{\mathbf{a}^{\H}\mathbf{R}^{-1}\mathbf{a}}.
\end{equation}
This solution presumes exact knowledge of the steering vector. It is, however,
acutely sensitive to array modeling errors arising from propagation uncertainty,
calibration inaccuracy, channel disturbance, and directional aiming error
\citep{vural1979effects,li2003robust,vorobyov2003robust,hassanien2008robust}.
Such errors produce steering vector mismatch and can degrade performance
severely, which motivates the RAB research
\citep{li2005robust,khabbazibasmenj2012robust,vorobyov2013principles}. Following
\citep{cox1987robust,feldman1994projection}, both mainstream lines of RAB
research decompose the true steering vector into its nominal value and a mismatch
quantity, $\mathbf{a}_{\mathrm{true}}=\mathbf{a}+\bm{\delta}$. They
differ in what they do next, and we review them in turn.

\subsection{Worst-Case Formulations}

The first line constrains $\bm{\delta}$ directly and comprises the
worst-case-constrained and probability-constrained approaches
\citep{hassanien2008robust,gershman2010convex}. The worst-case-constrained
approach \citep{vorobyov2003robust,lorenz2005robust,vorobyov2004adaptive,gershman2006robust}
confines $\bm{\delta}$ to the spherical uncertainty set
$\left\{ \bm{\delta}\,|\,\left\Vert \bm{\delta}\right\Vert _{2}\leq\varepsilon\right\}$
with uncertainty level $\varepsilon$, and requires the response magnitude to be
maintained for every mismatch in that set:

\begin{equation}
\min_{\left\Vert \bm{\delta}\right\Vert _{2}\leq\varepsilon}\left|\mathbf{w}^{\H}\left(\mathbf{a}+\bm{\delta}\right)\right|\geq1.\label{eq:robust constraint worst-case}
\end{equation}
By the triangle and Cauchy--Schwarz inequalities,

\begin{equation}
\min_{\left\Vert \bm{\delta}\right\Vert _{2}\leq\varepsilon}\left|\mathbf{w}^{\H}\left(\mathbf{a}+\bm{\delta}\right)\right|=\left|\mathbf{w}^{\H}\mathbf{a}\right|-\varepsilon\left\Vert \mathbf{w}\right\Vert _{2},
\end{equation}
with equality attained at
$\bm{\delta}=-\frac{\mathbf{w}}{\left\Vert \mathbf{w}\right\Vert _{2}}\varepsilon\exp\left(j\arg\left(\mathbf{w}^{\H}\mathbf{a}\right)\right)$
\citep{vorobyov2003robust}. A phase rotation renders
$\mathbf{w}^{\H}\mathbf{a}$ real-valued, so \eqref{eq:robust constraint worst-case}
becomes

\begin{equation}
\begin{gathered}\begin{alignedat}{1} & \mathbf{w}^{\H}\mathbf{a}\geq\varepsilon\left\Vert \mathbf{w}\right\Vert _{2}+1\\
 & \mathrm{Im}\left[\mathbf{w}^{\H}\mathbf{a}\right]=0.
\end{alignedat}
\end{gathered}
\label{eq:robust constraint worst-case-final}
\end{equation}
The worst case may be unduly conservative, since the least favorable mismatch is
unlikely to occur \citep{vorobyov2007relationship}. The probability-constrained
approach \citep{vorobyov2008relationship} instead models $\bm{\delta}$ as
a zero-mean random vector with covariance
$\bm{\Sigma}_{\bm{\delta}}$ and preserves the response with
probability exceeding a threshold $p_{\mathrm{th}}$:

\begin{equation}
\Pr\left\{ \left|\mathbf{w}^{\H}\left(\mathbf{a}+\bm{\delta}\right)\right|\geq1\right\} >p_{\mathrm{th}},\label{eq:robust constraint prob}
\end{equation}
which is tightly approximated by
\citep{vorobyov2013principles,vorobyov2008relationship}

\begin{equation}
\begin{gathered}\begin{alignedat}{1} & \mathbf{w}^{\H}\mathbf{a}\geq\varepsilon\left\Vert \bm{\Sigma}_{\bm{\delta}}^{1/2}\mathbf{w}\right\Vert _{2}+1\\
 & \mathrm{Im}\left[\mathbf{w}^{\H}\mathbf{a}\right]=0,
\end{alignedat}
\end{gathered}
\label{eq:robust constraint prob-final}
\end{equation}
with $\varepsilon=\sqrt{-\log\left(1-p_{\mathrm{th}}\right)}$ for Gaussian
$\bm{\delta}$ and $\varepsilon=1/\sqrt{1-p_{\mathrm{th}}}$ when the
distribution is unknown. Both \eqref{eq:robust constraint worst-case-final} and
\eqref{eq:robust constraint prob-final} are second-order cone (SOC) constraints,
so the two approaches share a common formulation:

\begin{equation}
\begin{aligned} & \underset{\mathbf{w}\in\mathbb{C}^{N}}{\mathsf{minimize}} &  & \mathbf{w}^{\H}\mathbf{R}\mathbf{w}\\
 & \mathsf{subject\ to} &  & \mathbf{w}^{\H}\mathbf{a}\geq\varepsilon\left\Vert \mathbf{A}\mathbf{w}\right\Vert _{2}+1\\
 &  &  & \mathrm{Im}\left[\mathbf{w}^{\H}\mathbf{a}\right]=0,
\end{aligned}
\label{eq:opt prob}
\end{equation}
where the affine parameter $\mathbf{A}\in\mathbb{C}^{M\times N}$ confines
$\bm{\delta}$ to an ellipsoidal uncertainty set. For the worst-case and
probability-constrained approaches, $\mathbf{A}$ is $\mathbf{I}$ and
$\bm{\Sigma}_{\bm{\delta}}^{1/2}$, respectively, so that $M=N$;
the reconstruction-based instantiation of Section \ref{sec:Connection-data-model} gives a
wider matrix, and the two cases are treated separately in Section
\ref{sec:Problem and Solution}. Problem \eqref{eq:opt prob} is a convex second-order
cone program (SOCP), studied in \citep{lorenz2005robust}; under limited snapshots
or low input SNR, diagonal loading \citep{li2003robust}
reduces the covariance estimation error. 
Further extensions include general convex and separable uncertainty models solved
by semidefinite programming (SDP) \citep{kim2008robust}, alongside \citep{vorobyov2004adaptive,shahbazpanahi2003robust,rong2005robust,rong2006robust,palomar2010convex}.

The formulation is convenient to solve but sensitive to the hyperparameters
$\varepsilon$, $\mathbf{A}$, and $\mathbf{R}$. Too small an $\varepsilon$ leaves
the true mismatch outside the uncertainty set; too large an $\varepsilon$ can
render the constraint set infeasible. The ellipsoidal structure depends on prior
information that may not be available. And $\mathbf{R}$, typically taken as the
sample covariance matrix, induces the self-cancellation effect when it contains a
strong desired signal component. This defect motivates the second line of
research.

\subsection{Reconstruction-Based Formulations}

The second line is represented by the interference covariance matrix
reconstruction (ICMR) methods \citep{gu2012robust}. Rather than guarding against
an uncertain steering vector, ICMR first estimates $\mathbf{R}_{i+n}$ from
$\hat{\mathbf{R}}$ through the Capon spatial spectrum, for instance as

\begin{equation}
\hat{\mathbf{R}}_{i+n}=\int_{\bar{\Theta}}\frac{\mathbf{a}\left(\theta\right)\mathbf{a}^{\H}\left(\theta\right)}{\mathbf{a}^{\H}\left(\theta\right)\hat{\mathbf{R}}^{-1}\mathbf{a}\left(\theta\right)}d\theta,
\end{equation}
where $\mathbf{a}\left(\theta\right)$ is the steering vector for hypothetical DOA
$\theta$ and $\bar{\Theta}$ the complement of the desired signal angular sector;
a scaled identity matrix may be added to avoid singularity and to account for the
noise covariance structure. The mismatch component orthogonal to $\mathbf{a}$,
denoted $\bm{\delta}_{\bot}$, is then chosen to maximize the beamformer
output power while keeping the corrected steering vector farther from the
interference region than the nominal one \citep{gu2012robust}; the parallel
component does not affect performance. This gives

\begin{equation}
\begin{aligned} & \underset{\bm{\delta}_{\bot}\in\mathbb{C}^{N}}{\mathsf{minimize}} &  & \left(\mathbf{a}+\bm{\delta}_{\bot}\right)^{\H}\mathbf{\hat{R}}_{i+n}^{-1}\left(\mathbf{a}+\bm{\delta}_{\bot}\right)\\
 & \mathsf{subject\ to} &  & \mathbf{a}^{\H}\bm{\delta}_{\bot}=0\\
 &  &  & \left(\mathbf{a}+\bm{\delta}_{\bot}\right)^{\H}\hat{\mathbf{R}}_{i+n}\left(\mathbf{a}+\bm{\delta}_{\bot}\right)\leq\mathbf{a}^{\H}\hat{\mathbf{R}}_{i+n}\mathbf{a},
\end{aligned}
\label{eq: icmr}
\end{equation}
a convex quadratically constrained quadratic program (QCQP). The objective and
the second constraint pull in opposite directions, and both are needed. Minimizing the objective maximizes the Capon output power of the
corrected steering vector, since that power is the reciprocal of the quadratic
form being minimized. Expanded in the eigenbasis of the reconstructed covariance,
however, the objective divides by the eigenvalues rather than multiplying by
them. It is made small by concentrating the energy of the corrected vector
along the directions of largest eigenvalue, which are the interferers. The
search for maximum output power, left to itself, would therefore drag the
correction onto the dominant interference, the opposite of what is wanted. The
second constraint is the counterweight: it weights those same directions by the
eigenvalues, penalizing precisely what the objective rewards, and caps the
interference-plus-noise power of the corrected steering vector at that of the
nominal one. It also renders the feasible set bounded, and it introduces no free
parameter, its threshold being read from $\mathbf{a}$ itself. With the optimal
$\bm{\delta}_{\bot}^{\star}$, the beamformer follows as

\begin{equation}
\mathbf{w}^{\star}=\frac{\mathbf{\hat{R}}_{i+n}^{-1}\left(\mathbf{a}+\bm{\delta}_{\bot}^{\star}\right)}{\left(\mathbf{a}+\bm{\delta}_{\bot}^{\star}\right)^{\H}\mathbf{\hat{R}}_{i+n}^{-1}\left(\mathbf{a}+\bm{\delta}_{\bot}^{\star}\right)}.\label{eq:icmr-beamformer}
\end{equation}
The approach is data-driven in that $\bm{\delta}_{\bot}^{\star}$ depends
on $\hat{\mathbf{R}}_{i+n}$, itself built from the snapshot data. Subsequent work
has refined the reconstruction, including estimation by integration over an
annulus uncertainty set \citep{huang2015robust} and the variants in
\citep{hassanien2008robust,huang2019new,zhu2020robust,luo2023urglq,yan2025analytically}. Experiments in this literature report that the reconstructed covariance and corrected steering
vector largely overcome desired signal cancellation, improving on the
worst-case approaches.

\subsection{Solution Methods and Their Cost}

Both formulations are convex, so the interior-point method attains the global
optimum \citep{vorobyov2003robust,hassanien2008robust,gu2012robust,huang2015robust},
usually through off-the-shelf solvers such as SeDuMi \citep{sturm1999using} or
MOSEK \citep{aps2019mosek}, at a complexity of
$\mathcal{O}\left(N^{3.5}\right)$ in both cases
\citep{lobo1998applications,gu2012robust,mohammadzadeh2023efficient}.

A more efficient alternative for the SOCP \eqref{eq:opt prob} is the
Lagrange-equation-based algorithm RMVB \citep{lorenz2005robust}. Mapping the
complex variable to its real and imaginary parts,
$\check{\mathbf{w}}=\left[\mathrm{Re}\left[\mathbf{w}\right]^{\T},\mathrm{Im}\left[\mathbf{w}\right]^{\T}\right]^{\T}$,
the optimal solution is expressed as

\begin{equation}
\check{\mathbf{w}}=-\zeta\left(\check{\mathbf{R}}+\zeta\check{\mathbf{Q}}\right)^{-1}\check{\mathbf{a}},
\end{equation}
with $\zeta$ the Lagrange multiplier and

\begin{equation}
\check{\mathbf{Q}}=\varepsilon^{2}\check{\mathbf{A}}^{\T}\check{\mathbf{A}}-\check{\mathbf{a}}\mathbf{\check{a}}^{\T},
\end{equation}
where $\check{\mathbf{R}}$, $\check{\mathbf{A}}$, and $\check{\mathbf{a}}$ are
the corresponding real-valued representations. Although $\check{\mathbf{w}}$ is
available in closed form, the optimal $\zeta$ must be computed numerically as a
zero of the scalar secular function

\begin{equation}
S\left(\zeta\right)=\zeta^{2}\sum_{n=1}^{N}\frac{\check{b}_{n}^{2}\gamma_{n}}{\left(1+\zeta\gamma_{n}\right)^{2}}-2\zeta\sum_{n=1}^{N}\frac{\check{b}_{n}^{2}}{1+\zeta\gamma_{n}}-1,
\end{equation}
obtained from the eigenvalue decomposition (EVD) of
$\left(\check{\mathbf{R}}^{-1/2}\right)^{\T}\check{\mathbf{Q}}\check{\mathbf{R}}^{-1/2}$.
Since $S\left(\zeta\right)$ is not necessarily monotonic and admits multiple
zeros, an interval isolating the desired one must first be established; RMVB does
so via a lower bound $\hat{\zeta}$ derived in \citep[Sec. II-B]{lorenz2005robust},
after which Newton-Raphson iteration converges to the unique zero in
$\left(\hat{\zeta},+\infty\right)$. The resulting complexity is
$\mathcal{O}\left(N^{3}\right)$, dominated by matrix factorization and EVD.

RMVB thus improves on the interior-point method, but three limitations remain.
It doubles the problem size by splitting the complex variable into real and
imaginary parts. It presumes $\check{\mathbf{R}}$ invertible, an assumption
embedded in the lower-bound analysis, so it does not necessarily apply when
$\check{\mathbf{R}}$ is rank deficient. And the construction of that bound is
itself intricate, being unavoidable only because $S\left(\zeta\right)$ may have
several real zeros.

In summary, the worst-case formulation admits a more efficient solution but
suffers desired signal cancellation, while the reconstruction-based formulation mitigates
that issue at the cost of requiring optimization software. The two are
treated throughout the literature as alternatives. The next section shows that
they are not.

\section{Connection Between the Worst-Case and Reconstruction-Based Formulations\label{sec:Connection-data-model}}

The two formulations of Section \ref{sec:Signal-Model-Existing} are conventionally
regarded as competing alternatives. The worst-case formulation is known to
suffer desired signal cancellation at high input SNR when the sample covariance
matrix contains a strong signal component, and the reconstruction-based formulation is generally
credited with repairing that deficiency. The natural conclusion is that the
worst-case approach has been superseded.

But that conclusion is incorrect. The two are not rival designs but the primal and
dual descriptions of a single problem: the reconstruction-based formulation
\eqref{eq: icmr} can be recast exactly into the worst-case formulation
\eqref{eq:opt prob} under a suitable identification of parameters. The practical
consequence is that desired signal cancellation is not intrinsic to the
worst-case-constrained formulation at all. It is a property of the parameters
with which that formulation is instantiated, and a choice guided by the
reconstruction-based formulation removes it.

We establish the equivalence in four steps: a change of variables that removes
the orthogonality constraint, a splitting that exposes a linear equality, the
dual problem, and a reduction of that dual to the worst-case form. 

We begin with the orthogonality between the nominal steering vector $\mathbf{a}$
and the mismatch component $\bm{\delta}_{\bot}$, which confines the
latter to the orthogonal complement of $\mathbf{a}$ and so admits the
parameterization $\bm{\delta}_{\bot}=\mathbf{S}_{a}\mathbf{h}$, where
$\mathbf{S}_{a}\in\mathbb{C}^{N\times\left(N-1\right)}$ is any basis for that
complement, i.e., $\mathbf{S}_{a}^{\H}\mathbf{a}=\mathbf{0}$. Substituting this parameterization
absorbs the equality constraint of \eqref{eq: icmr} into the variable itself and,
after a change of variables, leaves a quadratic objective subject to a single
norm-ball constraint.

\begin{lem}
\label{lem:transform}Let $\bar{\mathbf{R}}=\mathbf{S}_{a}^{\H}\hat{\mathbf{R}}_{i+n}\mathbf{S}_{a}$
and $\bar{\mathbf{a}}=\mathbf{S}_{a}^{\H}\hat{\mathbf{R}}_{i+n}\mathbf{a}$. Then the
quadratic constraint of the reconstruction-based formulation \eqref{eq: icmr} is
equivalently transformed into
\begin{equation}
\begin{aligned} & \left(\mathbf{a}+\bm{\delta}_{\bot}\right)^{\H}\hat{\mathbf{R}}_{i+n}\left(\mathbf{a}+\bm{\delta}_{\bot}\right)\leq\mathbf{a}^{\H}\hat{\mathbf{R}}_{i+n}\mathbf{a}\\
\Longleftrightarrow & \left\Vert \bar{\mathbf{R}}^{1/2}\mathbf{h}+\bar{\mathbf{R}}^{-1/2}\bar{\mathbf{a}}\right\Vert _{2}\leq\sqrt{\bar{\mathbf{a}}^{\H}\bar{\mathbf{R}}^{-1}\bar{\mathbf{a}}},
\end{aligned}
\label{eq:transform-constraint}
\end{equation}
or, more compactly,
\begin{equation}
\left\Vert \mathbf{q}\right\Vert _{2}\leq\varepsilon,\label{eq:transform-ball}
\end{equation}
and the objective into
\begin{equation}
\left(\mathbf{a}+\bm{\delta}_{\bot}\right)^{\H}\mathbf{\hat{R}}_{i+n}^{-1}\left(\mathbf{a}+\bm{\delta}_{\bot}\right)=\left(\mathbf{f}+\mathbf{F}\mathbf{q}\right)^{\H}\hat{\mathbf{R}}_{i+n}^{-1}\left(\mathbf{f}+\mathbf{F}\mathbf{q}\right),\label{eq:transform-objective}
\end{equation}
where
\begin{equation}
\begin{gathered}\mathbf{q}\triangleq\bar{\mathbf{R}}^{1/2}\mathbf{h}+\bar{\mathbf{R}}^{-1/2}\bar{\mathbf{a}},\qquad\varepsilon\triangleq\sqrt{\bar{\mathbf{a}}^{\H}\bar{\mathbf{R}}^{-1}\bar{\mathbf{a}}},\\
\mathbf{f}\triangleq\mathbf{a}-\mathbf{S}_{a}\bar{\mathbf{R}}^{-1}\bar{\mathbf{a}},\qquad\mathbf{F}\triangleq\mathbf{S}_{a}\bar{\mathbf{R}}^{-1/2}.
\end{gathered}
\label{eq:transform-defs}
\end{equation}
\end{lem}
\begin{IEEEproof}
See Appendix \ref{sec:Proof-of-Lemma transform} for the detailed
proof. 
\end{IEEEproof}
Lemma \ref{lem:transform} has a direct interpretation. The second constraint of
the reconstruction-based formulation, which requires the corrected steering vector to
carry no more interference-plus-noise power than the nominal one, becomes the
spherical uncertainty set $\left\Vert \mathbf{q}\right\Vert _{2}\leq\varepsilon$
of the worst-case-constrained approach, with
$\varepsilon=\sqrt{\bar{\mathbf{a}}^{\H}\bar{\mathbf{R}}^{-1}\bar{\mathbf{a}}}$.
In other words, the constraint that keeps the reconstruction-based method from
drifting onto the interference is the uncertainty set of the worst-case method
in different variables. Note also that $\varepsilon$ is not a free parameter
here: it is determined by the reconstructed covariance and the nominal steering
vector. This is the sense in which the reconstruction-based formulation fixes the
parameter that the worst-case formulation leaves to the user. We assume
$\varepsilon>0$, which holds exactly when $\bar{\mathbf{a}}\neq\mathbf{0}$, that is, when
$\mathbf{a}$ is not an eigenvector of $\hat{\mathbf{R}}_{i+n}$; otherwise the ball constraint
collapses to $\mathbf{q}=\mathbf{0}$ and the problem degenerates.

With Lemma \ref{lem:transform}, problem \eqref{eq: icmr} becomes 
\begin{equation}
\begin{aligned} & \underset{\mathbf{q}\in\mathbb{C}^{N-1}}{\mathsf{minimize}} &  & \left(\mathbf{f}+\mathbf{F}\mathbf{q}\right)^{\H}\hat{\mathbf{R}}_{i+n}^{-1}\left(\mathbf{f}+\mathbf{F}\mathbf{q}\right)\\
 & \mathsf{subject\thinspace to} &  & \left\Vert \mathbf{q}\right\Vert _{2}\leq\varepsilon,
\end{aligned}
\end{equation}
which, on introducing the auxiliary variable $\mathbf{z}$ defined by
$\hat{\mathbf{R}}_{i+n}^{1/2}\mathbf{z}=\mathbf{f}+\mathbf{F}\mathbf{q}$, is
equivalently written as 
\begin{equation}
\begin{array}{cl}
\underset{\mathbf{z}\in\mathbb{C}^{N},\mathbf{q}\in\mathbb{C}^{N-1}}{\mathsf{minimize}} & \left\Vert \mathbf{z}\right\Vert _{2}^{2}\\
\mathsf{subject\thinspace to} & \hat{\mathbf{R}}_{i+n}^{1/2}\mathbf{z}=\mathbf{f}+\mathbf{F}\mathbf{q}\\
 & \left\Vert \mathbf{q}\right\Vert _{2}\leq\varepsilon.
\end{array}
\label{eq:icmr-split}
\end{equation}
The substitution is exact, since $\hat{\mathbf{R}}_{i+n}\succ\mathbf{0}$, and it
serves two purposes: the quadratic objective becomes a plain squared norm, and
the coupling between $\mathbf{z}$ and $\mathbf{q}$ is carried by a linear
equality. Form \eqref{eq:icmr-split} is convenient because dualizing only that
equality leaves the ball constraint implicit. Attaching
the complex dual variable $\mathbf{w}$ to that equality gives the partial
Lagrangian
\begin{equation}
\mathcal{L}\left(\mathbf{z},\mathbf{q};\mathbf{w}\right)=\left\Vert \mathbf{z}\right\Vert _{2}^{2}+\mathrm{Re}\left[\mathbf{w}^{\H}\left(\mathbf{f}+\mathbf{F}\mathbf{q}-\hat{\mathbf{R}}_{i+n}^{1/2}\mathbf{z}\right)\right].
\end{equation}
Minimizing over $\mathbf{z}$ and $\mathbf{q}$ subject to the retained ball
constraint yields the dual function 
\begin{equation}
\begin{aligned}g\left(\mathbf{w}\right) & =\inf_{\left\Vert \mathbf{q}\right\Vert _{2}\leq\varepsilon}\left\Vert \mathbf{z}\right\Vert _{2}^{2}-\mathrm{Re}\left[\mathbf{w}^{\H}\left(\hat{\mathbf{R}}_{i+n}^{1/2}\mathbf{z}\right)\right]\\
 & \hspace{0.5cm}+\mathrm{Re}\left[\mathbf{w}^{\H}\mathbf{F}\mathbf{q}\right]+\mathrm{Re}\left[\mathbf{w}^{\H}\mathbf{f}\right]\\
 & =-\frac{1}{4}\mathbf{w}^{\H}\hat{\mathbf{R}}_{i+n}\mathbf{w}-\varepsilon\left\Vert \mathbf{F}^{\H}\mathbf{w}\right\Vert _{2}+\mathrm{Re}\left[\mathbf{w}^{\H}\mathbf{f}\right],
\end{aligned}
\end{equation}
where the minimizers are $\mathbf{z}=\frac{1}{2}\hat{\mathbf{R}}_{i+n}^{1/2}\mathbf{w}$,
from the zero-gradient condition, and
$\mathbf{q}=-\varepsilon\left(\mathbf{F}^{\H}\mathbf{w}\right)/\left\Vert \mathbf{F}^{\H}\mathbf{w}\right\Vert _{2}$,
which attains the Cauchy--Schwarz bound on the ball. The dual problem is
therefore 
\begin{equation}
\begin{aligned} & \underset{\mathbf{w}\in\mathbb{C}^{N}}{\mathsf{maximize}} &  & -\frac{1}{4}\mathbf{w}^{\H}\hat{\mathbf{R}}_{i+n}\mathbf{w}-\varepsilon\left\Vert \mathbf{F}^{\H}\mathbf{w}\right\Vert _{2}+\mathrm{Re}\left[\mathbf{w}^{\H}\mathbf{f}\right].\end{aligned}
\label{eq:data-driven dual}
\end{equation}
Problem \eqref{eq:data-driven dual} is a maximization with a norm term and is
not yet in the worst-case form. The remaining step separates the scale of
$\mathbf{w}$ from its direction. Writing any $\mathbf{w}\in\mathbb{C}^{N}$ as
$\mathbf{w}=\beta\bar{\mathbf{w}}$ with $\beta\geq0$, the objective becomes a
concave quadratic in $\beta$, whose maximizer is available in closed form. What
remains is a problem in $\bar{\mathbf{w}}$ alone. Its objective is a ratio of
two second-order homogeneous forms and therefore invariant to the scale of
$\bar{\mathbf{w}}$, which is what permits the normalization used below. 
\begin{lem}
\label{lem:qcqp2socp}The optimal solution $\mathbf{w}^{\star}$ of the dual
problem \eqref{eq:data-driven dual} is a positive multiple of the optimal
solution $\bar{\mathbf{w}}^{\star}$ of
\begin{equation}
\begin{aligned} & \underset{\bar{\mathbf{w}}\in\mathbb{C}^{N}}{\mathsf{minimize}} &  & \bar{\mathbf{w}}^{\H}\hat{\mathbf{R}}_{i+n}\bar{\mathbf{w}}\\
 & \mathsf{subject\thinspace to} &  & \mathrm{Re}\left[\bar{\mathbf{w}}^{\H}\mathbf{f}\right]-\varepsilon\left\Vert \mathbf{F}^{\H}\bar{\mathbf{w}}\right\Vert _{2}=1,
\end{aligned}
\end{equation}
and the scaling factor is $\beta^{\star}=\frac{2\left(\mathrm{Re}\left[\left(\bar{\mathbf{w}}^{\star}\right)^{\H}\mathbf{f}\right]-\varepsilon\left\Vert \mathbf{F}^{\H}\bar{\mathbf{w}}^{\star}\right\Vert _{2}\right)}{\left(\bar{\mathbf{w}}^{\star}\right)^{\H}\hat{\mathbf{R}}_{i+n}\bar{\mathbf{w}}^{\star}}=\frac{2}{\left(\bar{\mathbf{w}}^{\star}\right)^{\H}\hat{\mathbf{R}}_{i+n}\bar{\mathbf{w}}^{\star}}$. 
\end{lem}
\begin{IEEEproof}
See Appendix \ref{sec:Proof-of-Lemma qcqp2socp} for the detailed
proof. 
\end{IEEEproof}
The normalized problem of Lemma \ref{lem:qcqp2socp} is feasible, although this
is not obvious: the dual problem \eqref{eq:data-driven dual} is unconstrained,
whereas the equality constraint requires some $\bar{\mathbf{w}}$ to attain the
value one. Such a $\bar{\mathbf{w}}$ exists because the left-hand side of the
constraint is positively homogeneous of degree one in $\bar{\mathbf{w}}$ and
equals $\left\Vert \mathbf{a}\right\Vert _{2}^{2}>0$ at
$\bar{\mathbf{w}}=\mathbf{a}$: indeed $\mathbf{F}^{\H}\mathbf{a}=\mathbf{0}$ and
$\mathrm{Re}\left[\mathbf{a}^{\H}\mathbf{f}\right]=\left\Vert \mathbf{a}\right\Vert _{2}^{2}$
since $\mathbf{S}_{a}^{\H}\mathbf{a}=\mathbf{0}$. Scaling $\mathbf{a}$ by
$\left\Vert \mathbf{a}\right\Vert _{2}^{-2}$ therefore gives a feasible point. The
same positivity at $\bar{\mathbf{w}}=\mathbf{a}$ is what rules out the degenerate
maximizers in the proof of Lemma \ref{lem:qcqp2socp} in Appendix
\ref{sec:Proof-of-Lemma qcqp2socp}.

Both the relaxation of the equality to an inequality and the restriction to real
$\bar{\mathbf{w}}^{\H}\mathbf{f}$ are without loss of optimality (see the end of
Appendix \ref{sec:Proof-of-Lemma qcqp2socp}), so the dual problem
is equivalently 
\begin{equation}
\begin{aligned} & \underset{\bar{\mathbf{w}}\in\mathbb{C}^{N}}{\mathsf{minimize}} &  & \bar{\mathbf{w}}^{\H}\hat{\mathbf{R}}_{i+n}\bar{\mathbf{w}}\\
 & \mathsf{subject\thinspace to} &  & \bar{\mathbf{w}}^{\H}\mathbf{f}-\varepsilon\left\Vert \mathbf{F}^{\H}\bar{\mathbf{w}}\right\Vert _{2}\geq1\\
 &  &  & \mathrm{Im}\left[\bar{\mathbf{w}}^{\H}\mathbf{f}\right]=0,
\end{aligned}
\label{eq:reduced-dual}
\end{equation}
which is precisely the worst-case formulation \eqref{eq:opt prob} under a
particular choice of its parameters. Collecting the steps gives the main result
of this section.
\begin{thm}
\label{thm:equivalence}Let $\mathbf{f}$, $\mathbf{F}$ and $\varepsilon>0$ be as in
Lemma \ref{lem:transform}. The reconstruction-based formulation \eqref{eq: icmr} and the
worst-case formulation \eqref{eq:opt prob} with the parameters
\begin{equation}
\mathbf{R}=\hat{\mathbf{R}}_{i+n},\qquad\mathbf{a}=\mathbf{f},\qquad\mathbf{A}=\mathbf{F}^{\H}\label{eq:identification}
\end{equation}
constitute a primal-dual pair. Consequently the two approaches share the common
formulation \eqref{eq:opt prob}; moreover, the optimal solution of
\eqref{eq:opt prob} under \eqref{eq:identification} is exactly the beamformer
\eqref{eq:icmr-beamformer} produced by the reconstruction-based method, and its
optimal value is the reciprocal of that of \eqref{eq: icmr}.
\end{thm}
\begin{IEEEproof}
See Appendix \ref{sec:Proof-of-Theorem equivalence} for the detailed proof.
\end{IEEEproof}
The two lines of research therefore converge to a unified optimization problem. Two consequences follow. The first is conceptual: because the reconstruction-based
formulation is a particular parameterization of the worst-case one, its
immunity to desired signal cancellation is inherited by the worst-case
formulation whenever the same parameters are used. The weakness attributed to
worst-case-constrained beamforming is a matter of how it is instantiated
rather than a limitation of the formulation. The second is
practical: the unified problem returns the beamformer \eqref{eq:icmr-beamformer}
itself, without first solving for the mismatch variable $\bm{\delta}_{\bot}$ and substituting
it, as the classical ICMR procedure requires. 

The two formulations share a mathematical form but not the same parameter
regimes, and the distinction matters for what follows. In the reconstruction-based case
the coefficient matrix must be invertible, since \eqref{eq: icmr} inverts it, so
the reconstructed $\hat{\mathbf{R}}_{i+n}$ may need diagonal loading; its affine
parameter $\mathbf{F}^{\H}$, on the other hand, has fewer rows than columns and
so lacks full column rank. The worst-case instantiation is the reverse: $\mathbf{A}$ has
full column rank, while $\mathbf{R}$ is the sample covariance matrix and may be
rank deficient under limited snapshots. Neither regime dominates the other, and a
solution scheme that covers only one of them covers only half of the unified
problem. We therefore treat both regimes of problem \eqref{eq:opt prob}, with
$r=\mathrm{rank}\left(\mathbf{A}\right)$: 1) $r=N$ ($\mathbf{A}$ of full column
rank) with $\mathbf{R}\succeq\mathbf{0}$, $\mathbf{R}\neq\mathbf{0}$; and 2)
$r<N$ ($\mathbf{A}$ without full column rank) with $\mathbf{R}\succ\mathbf{0}$.
The reconstruction-based instance of Theorem \ref{thm:equivalence} belongs to the
second regime, with $r=N-1$. 

\section{The DTPAK Closed-Form Solution\label{sec:Problem and Solution}}

In this section, we present the derivation of the closed-form
solution for the unified formulation of RAB, i.e., problem \eqref{eq:opt prob}.
The basic assumptions for $\mathbf{a}$ and $\varepsilon$ are $\mathbf{a}\neq\mathbf{0}$
and $\varepsilon>0$, and those for $\mathbf{R}$ and $\mathbf{A}$
are consistent with those discussed at the end of the previous section, with
$r=\mathrm{rank}\left(\mathbf{A}\right)\leq N$ throughout. The two regimes, $r=N$ with
$\mathbf{R}\succeq\mathbf{0}$ and $\mathbf{R}\neq\mathbf{0}$, and $r<N$ with
$\mathbf{R}\succ\mathbf{0}$, are treated together: the three stages below bring both to
the same reduced problem, and a single theorem states the solution. 

\subsection{Diagonalization Transform }

The first stage finds a change of variables $\mathbf{w}=\mathbf{T}\mathbf{v}$ that
does two things at once: it turns the ellipsoidal norm $\left\Vert \mathbf{A}\mathbf{w}\right\Vert _{2}$
of the constraint into a standard Euclidean norm, and it diagonalizes the quadratic
objective so that its variable elements fully decouple. 
\begin{lem}[(Simultaneous Diagonalization)]
\label{lem:reduced data}There exist an invertible matrix $\mathbf{T}\in\mathbb{C}^{N\times N}$
and a vector $\bm{\lambda}\in\mathbb{R}^{N}$ with $\bm{\lambda}\geq\mathbf{0}$
such that
\begin{equation}
\mathbf{T}^{\H}\mathbf{A}^{\H}\mathbf{A}\mathbf{T}=\mathrm{BlkDiag}\left(\{\mathbf{I}_{r},\mathbf{0}\}\right),\qquad\mathbf{T}^{\H}\mathbf{R}\mathbf{T}=\mathrm{Diag}\left(\bm{\lambda}\right),
\end{equation}
and $\bm{\lambda}>\mathbf{0}$ if $r<N$. In other words, $\left\Vert \mathbf{A}\mathbf{T}\mathbf{v}\right\Vert _{2}=\left\Vert \mathbf{v}_{1:r}\right\Vert _{2}$
with $\mathbf{v}_{1:r}$ the first $r$ entries of $\mathbf{v}$, and $\mathbf{T}$
diagonalizes $\mathbf{R}$. 
\end{lem}
\begin{IEEEproof}
See Appendix \ref{sec:Proof-of-Lemma reduced data} for the detailed
proof. 
\end{IEEEproof}
The proof is constructive, and the construction differs in the two regimes. 

\emph{Full column rank case} ($r=N$, which requires $M\geq N$). We perform Cholesky decomposition
on $\mathbf{A}^{\H}\mathbf{A}$ to find an invertible square matrix
$\mathbf{B}$ such that $\mathbf{B}^{\H}\mathbf{B}=\mathbf{A}^{\H}\mathbf{A}$.
Then $\left\Vert \mathbf{A}\mathbf{w}\right\Vert _{2}=\left\Vert \mathbf{B}\mathbf{w}\right\Vert _{2}$, and since $\mathbf{B}$ is invertible, $\tilde{\mathbf{w}}=\mathbf{B}\mathbf{w}$ is a
change of variable rather than a restriction. The purpose of the
factorization is now apparent: the ellipsoidal norm has become the Euclidean norm
$\left\Vert \tilde{\mathbf{w}}\right\Vert _{2}$ of the variable itself, and a sphere,
unlike an ellipsoid, is left invariant by the unitary transformation used next to
diagonalize the objective matrix $\tilde{\mathbf{R}}=\left(\mathbf{B}^{-1}\right)^{\H}\mathbf{R}\mathbf{B}^{-1}$.
We take its EVD $\tilde{\mathbf{R}}=\mathbf{U}\bm{\Lambda}\mathbf{U}^{\H}$
with $\mathbf{U}$ unitary and $\bm{\Lambda}=\mathrm{Diag}\left(\bm{\lambda}\right)$,
the components of $\bm{\lambda}$ in decreasing order. The rotation $\mathbf{v}=\mathbf{U}^{\H}\tilde{\mathbf{w}}$ leaves the
Euclidean norm unchanged, which is exactly what the whitening step secured. The result
is $\mathbf{T}=\mathbf{B}^{-1}\mathbf{U}$. 

\emph{Non-full column rank case} ($r<N$). We perform compact singular value decomposition (SVD)
$\mathbf{A}=\mathbf{U}_{A}\mathrm{Diag}\left(\mathbf{d}\right)\mathbf{V}_{1}^{\H}$,
with the singular values $\mathbf{d}$ elementwise nonzero and in decreasing order, and
complete the right singular vectors $\mathbf{V}_{1}$ with an orthonormal basis $\mathbf{V}_{2}$
of their orthogonal complement. Any $\mathbf{w}\in\mathbb{C}^{N}$ can be uniquely
decomposed as $\mathbf{w}=\tilde{\mathbf{B}}\tilde{\mathbf{w}}$, with
$\tilde{\mathbf{B}}=\left[\mathbf{V}_{1}\left[\mathrm{Diag}\left(\mathbf{d}\right)\right]^{-1},\mathbf{V}_{2}\right]$ and
$\tilde{\mathbf{w}}=\left[\tilde{\mathbf{w}}_{1}^{\T},\tilde{\mathbf{w}}_{2}^{\T}\right]^{\T}$,
$\tilde{\mathbf{w}}_{1}\in\mathbb{C}^{r}$, $\tilde{\mathbf{w}}_{2}\in\mathbb{C}^{N-r}$,
and then $\left\Vert \mathbf{A}\mathbf{w}\right\Vert _{2}=\left\Vert \tilde{\mathbf{w}}_{1}\right\Vert _{2}$:
the ellipsoidal norm collapses onto the first block alone, while the objective still involves
all $N$ components. Partition $\tilde{\mathbf{R}}=\tilde{\mathbf{B}}^{\H}\mathbf{R}\tilde{\mathbf{B}}$
accordingly into blocks $\tilde{\mathbf{R}}_{11}$, $\tilde{\mathbf{R}}_{12}$, $\tilde{\mathbf{R}}_{21}$, $\tilde{\mathbf{R}}_{22}$.
Eliminating the coupling block by LDL decomposition
leaves $\tilde{\mathbf{w}}_{1}$ untouched, so the norm is preserved, and diagonalizing the
two remaining blocks, namely the Schur complement
$\tilde{\mathbf{S}}_{11}=\tilde{\mathbf{R}}_{11}-\tilde{\mathbf{R}}_{12}\tilde{\mathbf{R}}_{22}^{-1}\tilde{\mathbf{R}}_{21}$
and $\tilde{\mathbf{R}}_{22}$, by the EVDs $\tilde{\mathbf{S}}_{11}=\mathbf{U}_{1}\bm{\Lambda}_{1}\mathbf{U}_{1}^{\H}$
and $\tilde{\mathbf{R}}_{22}=\mathbf{U}_{2}\bm{\Lambda}_{2}\mathbf{U}_{2}^{\H}$ with unitary
$\mathbf{U}_{1}$, $\mathbf{U}_{2}$ completes the transform:
\begin{equation}
\mathbf{T}=\tilde{\mathbf{B}}\left[\begin{array}{cc}
\mathbf{U}_{1} & \mathbf{0}\\
-\tilde{\mathbf{R}}_{22}^{-1}\tilde{\mathbf{R}}_{21}\mathbf{U}_{1} & \mathbf{U}_{2}
\end{array}\right],
\end{equation}
with $\bm{\Lambda}=\mathrm{BlkDiag}\left(\left\{ \bm{\Lambda}_{1},\bm{\Lambda}_{2}\right\} \right)=\mathrm{Diag}\left(\bm{\lambda}\right)$.

The matrix $\mathbf{T}$ is not unique, since the EVDs are determined only up to a rotation
within each eigenspace of a repeated eigenvalue and the Cholesky factor only up to a unitary left
factor; any admissible $\mathbf{T}$ will do, and Theorem \ref{thm:unified} below
holds for each of them. Fix one, and define
\begin{equation}
\mathbf{b}=\mathbf{T}^{\H}\mathbf{a}\ (\neq\mathbf{0}),\qquad\mathbf{c}=\left|\mathbf{b}\right|.\label{eq:b-c-def}
\end{equation}
Because $\mathbf{w}^{\H}\mathbf{a}=\mathbf{v}^{\H}\mathbf{b}$, the substitution $\mathbf{w}=\mathbf{T}\mathbf{v}$
recasts problem \eqref{eq:opt prob} as
\begin{equation}
\begin{aligned} & \underset{\mathbf{v}\in\mathbb{C}^{N}}{\mathsf{minimize}} &  & \sum_{n=1}^{N}\lambda_{n}\left|v_{n}\right|^{2}\\
 & \mathsf{subject\ to} &  & \mathbf{v}^{\H}\mathbf{b}\geq\varepsilon\left\Vert \mathbf{v}_{1:r}\right\Vert _{2}+1\\
 &  &  & \mathrm{Im}\left[\mathbf{v}^{\H}\mathbf{b}\right]=0.
\end{aligned}
\label{eq:opt prob v}
\end{equation}
Since $\mathbf{T}$ is invertible, the map $\mathbf{w}=\mathbf{T}\mathbf{v}$ is a bijection
between the feasible sets, and the two problems are equivalent. The constraint is now a
sphere constraint on the leading $r$ entries and the objective is separable across
the components of $\mathbf{v}$.

\subsection{Phase Alignment }

The objective of \eqref{eq:opt prob v} depends on $\mathbf{v}$ only through the magnitudes
$\left|v_{n}\right|$, while the phases enter the constraint through $\mathbf{v}^{\H}\mathbf{b}$
alone. Aligning the phase of each $v_{n}$ with that of $b_{n}$ makes $\mathbf{v}^{\H}\mathbf{b}$
real and as large as the magnitudes allow, which suggests passing to the real field. Define
\begin{equation}
h\left(\mathbf{u}\right)\triangleq\mathbf{c}^{\T}\mathbf{u}-\varepsilon\left\Vert \mathbf{u}_{1:r}\right\Vert _{2}-1,\qquad\mathbf{u}\in\mathbb{R}^{N},\label{eq:h-def}
\end{equation}
where $\mathbf{c}$ is defined in \eqref{eq:b-c-def}.
Consider
\begin{equation}
\begin{aligned} & \underset{\mathbf{u}\in\mathbb{R}^{N}}{\mathsf{minimize}} &  & \sum_{n=1}^{N}\lambda_{n}u_{n}^{2}\\
 & \mathsf{subject\ to} &  & h\left(\mathbf{u}\right)\geq0.
\end{aligned}
\label{eq:opt prob 3}
\end{equation}
\begin{lem}
\label{lem:phase align}Problem \eqref{eq:opt prob v} has an optimal solution if and
only if problem \eqref{eq:opt prob 3} has one, and the two optimal values coincide. If
$\mathbf{u}^{\star}$ is optimal for \eqref{eq:opt prob 3}, so is $\left|\mathbf{u}^{\star}\right|$,
and $\mathbf{v}^{\star}=\left|\mathbf{u}^{\star}\right|\odot\exp\left(j\arg\left(\mathbf{b}\right)\right)$
is optimal for \eqref{eq:opt prob v}. To ensure mathematical consistency, we define the
phases of zero-valued elements as zero; this has no effect, as the corresponding entry
$c_{n}$ then vanishes as well.
\end{lem}
\begin{IEEEproof}
See Appendix \ref{sec:Proof-of-Lemma phase align} for the detailed
proof. 
\end{IEEEproof}
Phase alignment thus reduces the complex problem to a real one with the same
objective and a constraint of the same form, now with the nonnegative coefficients
$\mathbf{c}$. 

\subsection{KKT Solution\label{subsec:KKT}}

The third stage solves problem \eqref{eq:opt prob 3}. Define
\begin{equation}
\begin{aligned}\mathcal{I}_{\lambda,0} & =\left\{ n\,|\,\lambda_{n}=0\right\} ,\quad\tau_{0}=\sum_{n\in\mathcal{I}_{\lambda,0}}c_{n}^{2},\\
\mathcal{I}_{c,+} & =\left\{ n\in\left\{ r+1,\ldots,N\right\} \,|\,c_{n}>0\right\} ,\quad\tau_{r}=\sum_{n=1}^{r}c_{n}^{2}.
\end{aligned}
\end{equation}
When $r<N$, $\mathcal{I}_{\lambda,0}=\emptyset$ and $\tau_{0}=0$; when $r=N$,
$\mathcal{I}_{c,+}=\emptyset$. In both cases $\tau_{0}\leq\tau_{r}$.

The function $h$ is concave, so problem \eqref{eq:opt prob 3} is convex, and Slater's
condition holds whenever the problem is feasible: for feasible $\mathbf{u}$ and $t>1$,
$h\left(t\mathbf{u}\right)=t\left(h\left(\mathbf{u}\right)+1\right)-1>0$. Its Lagrangian is
$\mathcal{L}\left(\mathbf{u};\mu\right)=\sum_{n=1}^{N}\lambda_{n}u_{n}^{2}-\mu h\left(\mathbf{u}\right)$
with $\mu\geq0$ the Lagrange multiplier. Hence $\mathbf{u}$ is optimal if and only if the
Karush-Kuhn-Tucker (KKT) conditions hold true, i.e., there is $\mu\geq0$ such that
$\mu h\left(\mathbf{u}\right)=0$, $h\left(\mathbf{u}\right)\geq0$ and
\begin{equation}
2\lambda_{n}u_{n}=\begin{cases}
\mu\left(c_{n}-\varepsilon s_{n}\right) & n=1,\ldots,r\\
\mu c_{n} & n=r+1,\ldots,N,
\end{cases}\label{eq:KKT subgradient}
\end{equation}
where $\mathbf{s}=\mathbf{u}_{1:r}/\left\Vert \mathbf{u}_{1:r}\right\Vert _{2}$
if $\mathbf{u}_{1:r}\neq\mathbf{0}$ and $\left\Vert \mathbf{s}\right\Vert _{2}\leq1$
otherwise (a subgradient of the $\ell_{2}$ norm at $\mathbf{u}_{1:r}=\mathbf{0}$). Writing the
condition with a subgradient treats the points where the norm is not differentiable on the
same footing as the others. Every KKT point falls in exactly one
of three cases, and each can be solved explicitly.
\begin{itemize}
\item \emph{Case A}: $\mu=0$ and $\mathbf{u}_{1:r}\neq\mathbf{0}$. Then \eqref{eq:KKT subgradient}
forces $u_{n}=0$ whenever $\lambda_{n}>0$, so $\mathbf{u}$ is supported on $\mathcal{I}_{\lambda,0}$,
which can be nonempty only if $r=N$. Primal feasibility then holds if and only if $\varepsilon^{2}<\tau_{0}$.
The optimal value is zero and $\mathbf{u}$ is not unique.
\item \emph{Case B}: $\mu>0$ and $\mathbf{u}_{1:r}\neq\mathbf{0}$. We introduce the auxiliary variable
$k=\mu\varepsilon/\left\Vert \mathbf{u}_{1:r}\right\Vert _{2}>0$, which fuses the primal variable with the
multiplier. Then \eqref{eq:KKT subgradient} gives $u_{n}=\mu c_{n}/\left(2\lambda_{n}+k\right)$ for $n\leq r$
and $u_{n}=\mu c_{n}/\left(2\lambda_{n}\right)$ for $n>r$, and the definition of $k$ reduces to the single
scalar equation $f\left(k\right)=\sum_{n=1}^{r}\left(\frac{c_{n}k}{2\lambda_{n}+k}\right)^{2}=\varepsilon^{2}$.
Each term of $f$ is nondecreasing in $k$ (and equals $c_{n}^{2}$ if $\lambda_{n}=0$), so $f$ is
nondecreasing from $f\left(0^{+}\right)=\tau_{0}$ to $f\left(+\infty\right)=\tau_{r}$. The
equation has a root, which is unique and found by bisection, if and only if $\tau_{0}<\varepsilon^{2}<\tau_{r}$.
The multiplier $\mu$ then follows in closed form from $h\left(\mathbf{u}\right)=0$.
\item \emph{Case C}: $\mathbf{u}_{1:r}=\mathbf{0}$, where the norm is not differentiable. The first $r$ equations
of \eqref{eq:KKT subgradient} read $c_{n}=\varepsilon s_{n}$ with $\left\Vert \mathbf{s}\right\Vert _{2}\leq1$,
which is possible if and only if $\varepsilon^{2}\geq\tau_{r}$. The remaining ones give
$u_{n}=\mu c_{n}/\left(2\lambda_{n}\right)$ for $n>r$, and $h\left(\mathbf{u}\right)=0$ determines
$\mu$ if and only if $\mathcal{I}_{c,+}\neq\emptyset$. This case is void when $r=N$.
\end{itemize}
The three regimes $\varepsilon^{2}<\tau_{0}$, $\tau_{0}<\varepsilon^{2}<\tau_{r}$ and $\varepsilon^{2}\geq\tau_{r}$
are disjoint, so which case applies is decided by $\varepsilon^{2}$ alone. The full analysis is given in the
proof of Theorem \ref{thm:unified}, Appendix \ref{sec:Proof-of-Theorem unified}.

\subsection{Closed-Form Result\label{subsec:Closed-Form Result}}

Assembling the three stages gives the closed-form solution of the unified
problem. Recall that $\mathbf{T}$, $\bm{\lambda}$, $\mathbf{b}$ and $\mathbf{c}$ come from the
Diagonalization Transform, and $\mathcal{I}_{\lambda,0}$, $\mathcal{I}_{c,+}$, $\tau_{0}$ and $\tau_{r}$ were
defined above.

\begin{thm}
\label{thm:unified}Under the assumptions of this section, problem \eqref{eq:opt prob}
has an optimal solution if and only if one of the conditions (i)--(iii) below
holds, and then an optimal solution is 
\begin{equation}
\mathbf{w}^{\star}=\mathbf{T}\left[\mathbf{u}^{\star}\odot\exp\left(j\arg\left(\mathbf{b}\right)\right)\right]
\end{equation}
with $\mathbf{u}^{\star}$ given by
\begin{enumerate}
\item[(i)] $\varepsilon^{2}<\tau_{0}$: $u_{n}^{\star}=0$ for $n\notin\mathcal{I}_{\lambda,0}$,
and $\left(u_{n}^{\star}\right)_{n\in\mathcal{I}_{\lambda,0}}\geq\mathbf{0}$
is any vector satisfying $\sum_{n\in\mathcal{I}_{\lambda,0}}c_{n}u_{n}^{\star}\geq\varepsilon\sqrt{\sum_{n\in\mathcal{I}_{\lambda,0}}\left(u_{n}^{\star}\right)^{2}}+1$;
\item[(ii)] $\tau_{0}<\varepsilon^{2}<\tau_{r}$:
\begin{equation}
u_{n}^{\star}=\begin{cases}
\frac{\mu c_{n}}{2\lambda_{n}+k} & n=1,\ldots,r\\
\frac{\mu c_{n}}{2\lambda_{n}} & n=r+1,\ldots,N,
\end{cases}
\end{equation}
where $k>0$ is the unique root of $f\left(k\right)=\sum_{n=1}^{r}\left(\frac{c_{n}k}{2\lambda_{n}+k}\right)^{2}=\varepsilon^{2}$,
found by bisection, and $\mu=\left[\sum_{n=1}^{r}\frac{2\lambda_{n}c_{n}^{2}}{\left(2\lambda_{n}+k\right)^{2}}+\sum_{n=r+1}^{N}\frac{c_{n}^{2}}{2\lambda_{n}}\right]^{-1}$;
\item[(iii)] $\varepsilon^{2}\geq\tau_{r}$ and $\mathcal{I}_{c,+}\neq\emptyset$:
$u_{n}^{\star}=0$ for $n=1,\ldots,r$ and $u_{n}^{\star}=\frac{c_{n}}{\lambda_{n}}/\left[\sum_{m=r+1}^{N}\frac{c_{m}^{2}}{\lambda_{m}}\right]$
for $n=r+1,\ldots,N$.
\end{enumerate}
When a solution exists, it is unique if and only if case (i) does not hold, that is, if and only if $\varepsilon^{2}>\tau_{0}$. Uniqueness is automatic when $\mathbf{R}\succ\mathbf{0}$, since then $\tau_{0}=0$. No solution exists exactly in the two remaining cases:
\begin{enumerate}
\item[(iv)] $\varepsilon^{2}\geq\tau_{r}$ and $\mathcal{I}_{c,+}=\emptyset$: the problem
is infeasible;
\item[(v)] $\varepsilon^{2}=\tau_{0}<\tau_{r}$ (possible only when $r=N$ and $\mathbf{R}$ is
rank-deficient): the problem is feasible but its infimum is not attained
(Example \ref{exa: def r nonexist - 1}).
\end{enumerate}
\end{thm}
\begin{IEEEproof}
See Appendix \ref{sec:Proof-of-Theorem unified} for the detailed proof. It
solves the KKT conditions of Section \ref{subsec:KKT} case by case, and settles
uniqueness with Lemma \ref{lem:unique criterion}, which is stated for problem
\eqref{eq:opt prob} itself.
\end{IEEEproof}
\begin{lem}
\label{lem:unique criterion}Suppose problem \eqref{eq:opt prob} has an optimal
solution, with optimal value $p^{\star}$, and let $\mu$ be the Lagrange multiplier
of problem \eqref{eq:opt prob 3} in \eqref{eq:KKT subgradient}. Then $p^{\star}=\mu/2$.
If $p^{\star}=0$, the solution is not unique. If $p^{\star}>0$, it is unique.
\end{lem}
\begin{IEEEproof}
See Appendix \ref{sec:Proof-of-Theorem unified} for the detailed proof. 
\end{IEEEproof}
The multiplier is $\mu=0$ in case (i) of Theorem \ref{thm:unified}, the value displayed
there in case (ii), and $\mu=\left[\sum_{n=r+1}^{N}c_{n}^{2}/\left(2\lambda_{n}\right)\right]^{-1}$
in case (iii). It is thus not merely an intermediate quantity of the KKT solution: it fixes the
optimal value, and zero optimal value is exactly the non-unique case (i).

Because the theorem holds for every admissible $\mathbf{T}$, and uniqueness is a property of
problem \eqref{eq:opt prob} alone, the optimal solution when unique does not depend on which
$\mathbf{T}$ is used to compute it. 

Every quantity in Theorem \ref{thm:unified} is obtained from a matrix
factorization or a scalar bisection; no iterative optimization over the weight
vector is involved at any point. The scheme, comprising the Diagonalization
Transform, Phase Alignment and KKT Solution stages, is named DTPAK. 
\begin{rem}
In case (i) of Theorem \ref{thm:unified}, $\mathbf{u}^{\star}$ is non-unique. One
optimal choice is $u_{n}^{\star}=\frac{c_{n}}{\tau_{0}-\varepsilon\sqrt{\tau_{0}}}$
for $n\in\mathcal{I}_{\lambda,0}$ (cf. the proof in Appendix \ref{sec:Proof-of-Lemma exist eqv}),
and for any optimal $\mathbf{u}^{\star}$ and $t>1$, $t\mathbf{u}^{\star}$ is also optimal. 
\end{rem}
\begin{rem}
The overall computational complexity is of the same order as RMVB
($\mathcal{O}\left(N^{3}\right)$) and is dominated by the matrix
inversion, EVD, and/or SVD operations. 
\end{rem}

\section{Insight into the Solution: Three Regimes of the Uncertainty Radius \label{sec:Existence-and-Uniqueness}}

Theorem \ref{thm:unified} already settles when a solution exists and when it is unique. This
section does not add to it; it reads those conditions in terms of the uncertainty radius $\varepsilon$, to show what
they mean and how they look in examples. Two thresholds, both computed from the data,
\[
\tau_{0}=\sum_{n\in\mathcal{I}_{\lambda,0}}c_{n}^{2}\qquad\text{and}\qquad\tau_{r}=\sum_{n=1}^{r}c_{n}^{2},
\]
split the range of $\varepsilon^{2}$ into three regimes. For $\varepsilon^{2}<\tau_{0}$ case (i) of the
theorem applies. For $\tau_{0}<\varepsilon^{2}<\tau_{r}$ case (ii) applies, with the boundary point
$\varepsilon^{2}=\tau_{0}$ being the no-solution case (v). For $\varepsilon^{2}\geq\tau_{r}$ the solution is
case (iii) if $\mathcal{I}_{c,+}\neq\emptyset$, and otherwise the problem is infeasible, which is the
no-solution case (iv). Figure~\ref{fig:regimes} shows them on a small example.

\begin{figure}[t]
\centering
\includegraphics[width=\columnwidth]{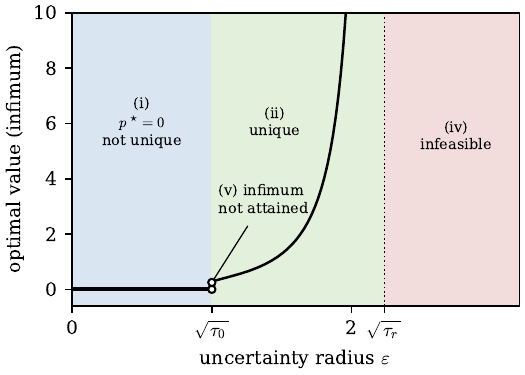}
\caption{Optimal value of problem \eqref{eq:opt prob} versus $\varepsilon$ for
$\mathbf{R}=\mathrm{Diag}\left(\left[1,0\right]^{\T}\right)$, $\mathbf{A}=\mathbf{I}$ and
$\mathbf{a}=\left[2,1\right]^{\T}$, so that $\tau_{0}=1$ and $\tau_{r}=5$. Case (iii) does not occur
here, since it requires $r<N$.}
\label{fig:regimes}
\end{figure}

\subsection{Existence Study}

Everything turns on the scalar $k>0$ of case B of Section \ref{subsec:KKT}, the root of
$f\left(k\right)=\sum_{n=1}^{r}\left(\frac{c_{n}k}{2\lambda_{n}+k}\right)^{2}=\varepsilon^{2}$.
Since $f$ increases from $\tau_{0}$ to $\tau_{r}$, the root exists exactly when
$\tau_{0}<\varepsilon^{2}<\tau_{r}$; outside this interval a solution can still exist, but only
through the other cases, namely A (when $\varepsilon^{2}<\tau_{0}$) or C (when $\varepsilon^{2}\geq\tau_{r}$,
where $\mathbf{u}_{1:r}=\mathbf{0}$). Specializing Theorem \ref{thm:unified}
to the regimes of Section \ref{sec:Connection-data-model} gives the following.
An optimal solution to problem \eqref{eq:opt prob} exists if and only if
\begin{itemize}
\item $\varepsilon^{2}<\tau_{r}$, when $\mathbf{A}$ has full
column rank and $\mathbf{R}\succ\mathbf{0}$ (here $\tau_{0}=0$);
\item $\varepsilon^{2}<\tau_{r}$ and
$\varepsilon^{2}\neq\tau_{0}$, when $\mathbf{A}$ has
full column rank and $\mathbf{R}$ is rank-deficient;
\item $\varepsilon^{2}<\tau_{r}$ or $\mathcal{I}_{c,+}\neq\emptyset$, when $\mathbf{A}$ does not have
full column rank (then $\mathbf{R}\succ\mathbf{0}$ and $\tau_{0}=0$).
\end{itemize}
Indeed, cases (i) and (v) of Theorem \ref{thm:unified} are void when $r<N$, and case (iii) is void when $r=N$.
The upper limit has a direct reading: once $\varepsilon^{2}$ reaches $\tau_{r}$ and no direction in the null space of $\mathbf{A}$ is available
($\mathcal{I}_{c,+}=\emptyset$), the
uncertainty set is too large for any beamformer to satisfy the distortionless
constraint, and the feasible set is empty. In the first simulation example of Vorobyov et al.\ \cite{vorobyov2003robust}, the array has $10$ sensors (our $N$). The nominal steering vector is normalized so that $\|\bar{\mathbf{a}}\|^{2}=10$ and is exactly known. The sample covariance matrix is built from $100$ snapshots and is hence highly likely to be full rank. This is the case $r=N$, $\mathbf{R}\succ\mathbf{0}$ of Theorem \ref{thm:unified} with $\mathbf{A}=\mathbf{I}$ and $\tau_{r}=\sum_{n}c_{n}^{2}=\|\bar{\mathbf{a}}\|^{2}=10$, so a solution exists if and only if $\varepsilon<\sqrt{10}\approx3.16$. The sweep of $\varepsilon$ in Figure 3 of \cite{vorobyov2003robust} ends at $3$, just below this limit; a value such as $3.5$ would have been infeasible.

The excluded value
$\varepsilon^{2}=\tau_{0}$ in the rank-deficient case is
of a different and more interesting kind. There the feasible set is \emph{not}
empty, yet no optimal solution exists: the infimum is approached but never
attained. Feasibility is therefore not sufficient for solvability, a distinction
that a numerical solver will not report, since it returns an approximately
optimal point either way. The following example exhibits the phenomenon with
$\mathbf{A}$ invertible and $\mathbf{R}$ rank-deficient. 

\begin{example}
\label{exa: def r nonexist - 1}Let $\mathbf{R}=\mathrm{Diag}\left(\left[1,0\right]^{\T}\right)$,
$\mathbf{A}=\mathbf{I}$, $\varepsilon=2$, and $\mathbf{a}=\left[1,2\right]^{\T}=\mathbf{c}$.
Problem \eqref{eq:opt prob} becomes 
\begin{equation}
\begin{array}{cl}
\underset{\mathbf{w}=\left[w_{1},w_{2}\right]^{\T}\in\mathbb{R}^{2}}{\mathsf{minimize}} & w_{1}^{2}\\
\mathsf{subject\ to} & w_{1}+2w_{2}\geq2\sqrt{w_{1}^{2}+w_{2}^{2}}+1\\
 & \mathrm{Im}\left[w_{1}+2w_{2}\right]=0.
\end{array}
\end{equation}
This example satisfies $\varepsilon^{2}=\tau_{0}=\sum_{n\in\mathcal{I}_{\lambda,0}}c_{n}^{2}$
with $\mathcal{I}_{\lambda,0}=\left\{ 2\right\} $ and has a feasible point
$\left[2,4\right]^{\T}$. To minimize the objective, we first calculate
a lower bound for $w_{1}$:

\begin{equation}
\begin{aligned}w_{1}\geq & 2\sqrt{w_{1}^{2}+w_{2}^{2}}-2w_{2}+1\\
\geq & 2\left(\sqrt{w_{1}^{2}+w_{2}^{2}}-\left|w_{2}\right|\right)+1\\
= & 2\frac{w_{1}^{2}}{\sqrt{w_{1}^{2}+w_{2}^{2}}+\left|w_{2}\right|}+1\geq1,
\end{aligned}
\end{equation}
which indicates that $w_{1}$ can be expressed as $1+\delta$ with
$\delta\geq0$. Next, we show that the increment $\delta$ can be
any positive value but is never zero. We replace $w_{1}$ with
$1+\delta$ and solve the inequality again: 
\begin{equation}
\begin{aligned} & 1+\delta+2w_{2}\geq2\sqrt{\left(1+\delta\right)^{2}+w_{2}^{2}}+1\\
\Longleftrightarrow & \begin{cases}
\left(\delta+2w_{2}\right)^{2}\geq4\left(\left(1+\delta\right)^{2}+w_{2}^{2}\right)\\
\delta+2w_{2}\geq0
\end{cases}\\
\Longleftrightarrow & \begin{cases}
4\delta w_{2}\geq3\delta^{2}+8\delta+4\\
2w_{2}\geq-\delta
\end{cases}
\end{aligned}
\end{equation}
When $\delta=0$, the inequalities are never satisfied under a finite
$w_{2}$. For any $\delta>0$, there always exists a sufficiently
large $w_{2}$ that satisfies both inequalities. Hence, the optimal
value $w_{1}^{2}=1$ cannot be strictly achieved and the optimal
solution $w_{1}=1$ can never be reached, proving nonexistence of
a finite solution. 
\end{example}

\subsection{Uniqueness Discussion}

Strict convexity of the objective, that is, $\mathbf{R}\succ\mathbf{0}$, is
sufficient for the optimal solution to be unique whenever it exists, and this is
the standard argument. It is not necessary. When $\mathbf{R}$ is rank-deficient,
the objective is merely convex, and yet the solution can still be unique because the
constraint supplies the curvature the objective lacks. Lemma \ref{lem:unique criterion}
makes this precise: when a solution exists, it is unique if and only if its optimal
value is positive, which by Theorem \ref{thm:unified} means exactly $\varepsilon^{2}>\tau_{0}$.
For $r=N$ the argument runs as follows. Two optimal solutions cannot differ by a positive multiple, since the
objective is quadratic and positive; and they cannot be non-proportional, since then the
midpoint would lie strictly inside the active constraint, where the objective can still be
lowered. By contrast, when $\varepsilon^{2}<\tau_{0}$ the objective vanishes on a feasible
point and any scaling of it by $t>1$ is optimal too. 

The statement concerns problem \eqref{eq:opt prob} itself, so the ambiguities in the
matrix factorization and in the EVD, which stem from unitary invariance and repeated
eigenvalues, have no bearing on it. In terms of the parameters, an optimal solution is unique if and only if
\begin{itemize}
\item $\varepsilon^{2}<\tau_{r}$, when $\mathbf{A}$ has full
column rank and $\mathbf{R}\succ\mathbf{0}$;
\item $\tau_{0}<\varepsilon^{2}<\tau_{r}$,
when $\mathbf{A}$ has full column rank and $\mathbf{R}$ is rank-deficient;
\item $\varepsilon^{2}<\tau_{r}$ or $\mathcal{I}_{c,+}\neq\emptyset$, when $\mathbf{A}$ does not have
full column rank,
\end{itemize}
that is, in every instance in which a solution exists apart from case (i). 
The middle case is the noteworthy one: it admits a unique solution although the
objective is not strictly convex, provided $\varepsilon$ is confined to that
interval.
\begin{example}
\label{exa: def r unique}Let $\mathbf{R}=\mathrm{Diag}\left(\left[1,0\right]^{\T}\right)$,
$\mathbf{A}=\mathbf{I}$, $\varepsilon=3/\sqrt{2}$, and $\mathbf{a}=\left[1,2\right]^{\T}=\mathbf{c}$.
Problem \eqref{eq:opt prob} becomes 
\begin{equation}
\begin{array}{cl}
\underset{\mathbf{w}=\left[w_{1},w_{2}\right]^{\T}\in\mathbb{R}^{2}}{\mathsf{minimize}} & w_{1}^{2}\\
\mathsf{subject\ to} & w_{1}+2w_{2}\geq\frac{3}{\sqrt{2}}\sqrt{w_{1}^{2}+w_{2}^{2}}+1\\
 & \mathrm{Im}\left[w_{1}+2w_{2}\right]=0.
\end{array}
\end{equation}
This example satisfies $\tau_{0}<\varepsilon^{2}<\tau_{r}$.
The problem is feasible with an interior point $\left[2\sqrt{6},6\sqrt{2}\right]^{\T}$
and the optimal solution is $\mathbf{w}^{\star}=\left[3.4142,9.6569\right]^{\T}$
according to Sec. \ref{subsec:Closed-Form Result}. 
\end{example}

\section{Numerical Simulations \label{sec:Numerical-Simulations}}

The simulations serve two purposes. The first is to test Theorem
\ref{thm:equivalence} empirically: instantiating the unified formulation with the
reconstruction-based parameters should reproduce the ICMR beamformer exactly, and with the
worst-case parameters the beamformer of Vorobyov et al. \citep{vorobyov2003robust}, regardless of the solution algorithm in either case. The second is to compare solution methods on the same
problem, confirming that DTPAK reaches the same optimum as a general-purpose
solver at lower cost. All simulations run on a PC with a 2.90 GHz i7-10700 CPU
and 16.0 GB RAM. For the beamformer, we consider the original ICMR method of \citep{gu2012robust}
and follow its implementation. For the solution algorithm, we
consider the interior-point method implemented with the off-the-shelf
solver MOSEK and the RMVB algorithm of \citep{lorenz2005robust}. 

\textbf{Default Simulation Settings.} We consider a ULA equipped with
$N=40$ omnidirectional antenna elements at half-wavelength spacing.
The additive background noise follows a zero-mean complex circularly
symmetric Gaussian distribution. The steering vector $\mathbf{a}\left(\theta\right)$
is expressed as $\mathbf{a}\left(\theta\right)=\left[1,\exp\left(-j\pi\sin\theta\right),\ldots,\exp\left(-j\pi\left(N-1\right)\sin\theta\right)\right]^{\T}$.
The desired signal arrives from the direction $\theta_{s}=16^{\lyxmathsym{\textdegree}}$
and the estimated direction $\bar{\theta}_{s}$ is $15^{\lyxmathsym{\textdegree}}$.
The input SNR is $10$ dB. Two interference signals originate from
the DOAs $\theta_{i,1}=-42^{{^\circ}}$ and $\theta_{i,2}=-11^{{^\circ}}$,
and the estimated directions are $\bar{\theta}_{i,1}=-45^{{^\circ}}$
and $\bar{\theta}_{i,2}=-10^{{^\circ}}$. The interference-to-noise
ratio is $25$ dB. The number of snapshots $T$ is $5N$. In the reconstruction
of $\mathbf{R}_{i+n}$, the desired signal angular sector is $\left[\bar{\theta}_{s}-6^{\lyxmathsym{\textdegree}},\ensuremath{\bar{\theta}_{s}}+\ensuremath{6^{\lyxmathsym{\textdegree}}}\right]=\ensuremath{\left[9^{\lyxmathsym{\textdegree}},2\ensuremath{1^{\lyxmathsym{\textdegree}}}\right]}$
and its complement set is $\bar{\Theta}=\left[-90^{\lyxmathsym{\textdegree}},\ensuremath{9^{\lyxmathsym{\textdegree}}}\right)\ensuremath{\cup\left(21^{\lyxmathsym{\textdegree}},9\ensuremath{0^{\lyxmathsym{\textdegree}}}\right]}$.
We present the numerical results through three representative mismatch
examples. The reported results are all averaged over $200$
randomized instances. 

Both beamformers come from the unified formulation \eqref{eq:opt prob} under the
parameter choices of Theorem \ref{thm:equivalence}. For the worst-case parameters,
$\mathbf{R}$
is chosen as the sample covariance matrix $\hat{\mathbf{R}}$, $\mathbf{a}=\mathbf{a}\left(\bar{\theta}_{s}\right)$,
$\mathbf{A}=\mathbf{I}$, and $\varepsilon=\sqrt{0.98N}$. The reconstruction-based parameters, which are
\eqref{eq:identification} written out, are $\mathbf{R}=\mathbf{R}_{i+n}$,
$\mathbf{a}=\mathbf{a}\left(\bar{\theta}_{s}\right)-\mathbf{S}_{a}\bar{\mathbf{R}}^{-1}\bar{\mathbf{a}}$,
$\mathbf{A}=\bar{\mathbf{R}}^{-1/2}\mathbf{S}_{a}^{\H}$, and $\varepsilon=\sqrt{\bar{\mathbf{a}}^{\H}\bar{\mathbf{R}}^{-1}\bar{\mathbf{a}}}$
where $\bar{\mathbf{R}}=\mathbf{S}_{a}^{\H}\hat{\mathbf{R}}_{i+n}\mathbf{S}_{a}$,
$\bar{\mathbf{a}}=\mathbf{S}_{a}^{\H}\hat{\mathbf{R}}_{i+n}\mathbf{a}\left(\bar{\theta}_{s}\right)$,
and $\mathbf{S}_{a}$ is the orthogonal complement of $\mathbf{a}\left(\bar{\theta}_{s}\right)$.
When the number of snapshots is less than $N$, $0.02\mathbf{I}$
could be added to $\hat{\mathbf{R}}$ for numerical stability. 

\subsection{Signal Look Direction Mismatch}

In the first example, the true DOAs of the desired and interference
signals are assumed to follow a truncated Gaussian distribution with
zero mean, standard deviation $\sigma=2.5/3$, and truncation bounds
at $\pm3\sigma$. The true DOAs remain fixed within each snapshot
block but vary across different instances. 

Figures \ref{fig:SINR vs SNR-1} and \ref{fig:SINR vs sample-1} display
the output SINR versus the input SNR and the number of snapshots,
respectively. The ICMR beamformer is recovered exactly by the unified formulation under the
reconstruction-based parameters, as Theorem \ref{thm:equivalence} predicts, and the
output SINR is very close to the optimal bound. The worst-case-constrained
beamformer never outperforms the reconstruction-based counterpart,
especially at high input SNR, and the SINR loss is asymptotically
10 dB. Figure \ref{fig:SINR vs sample-1} also shows
a clear gap across the entire range of snapshot numbers. That gap is not a
limitation of the formulation, however: switching to the reconstruction-based parameters
removes it, with no change of algorithm. This is the practical content of the
equivalence.
Whichever parameters are used, different solution methods give nearly
identical beamforming performance, which is a sanity check
for optimality of the compared algorithms. 

\begin{figure}[tbh]
\begin{centering}
\includegraphics[width=0.8\columnwidth]{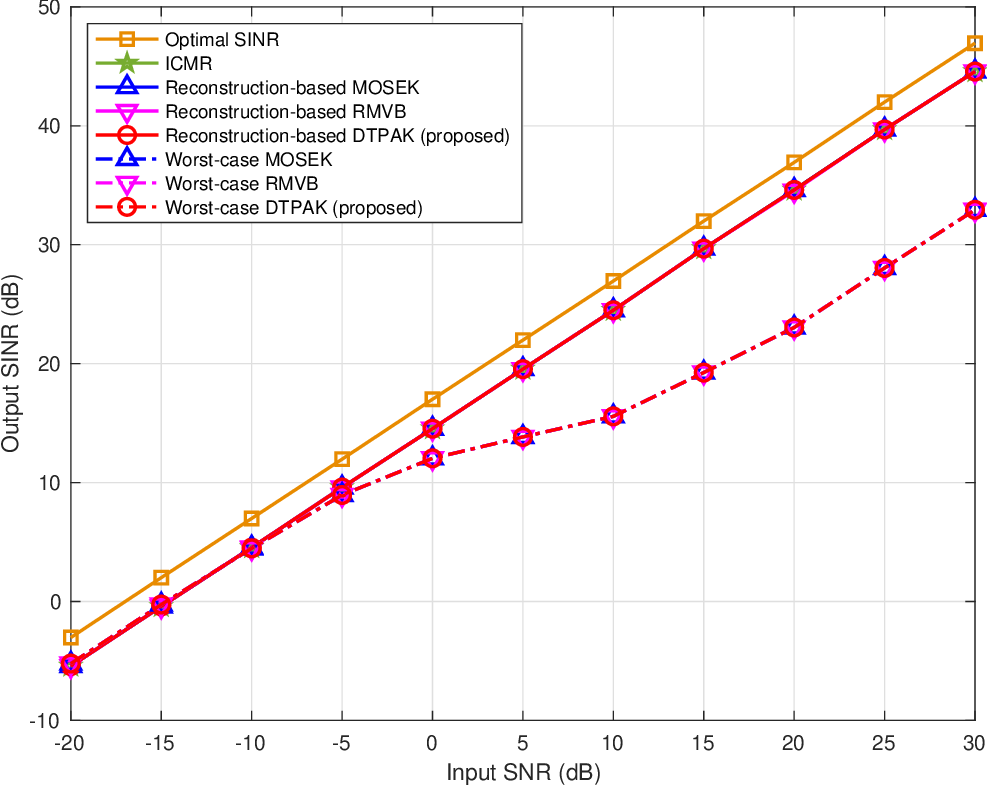}
\par\end{centering}
\caption{\label{fig:SINR vs SNR-1}Output SINR versus input SNR in the case
of signal look direction mismatch.}
\end{figure}

\begin{figure}[tbh]
\begin{centering}
\includegraphics[width=0.8\columnwidth]{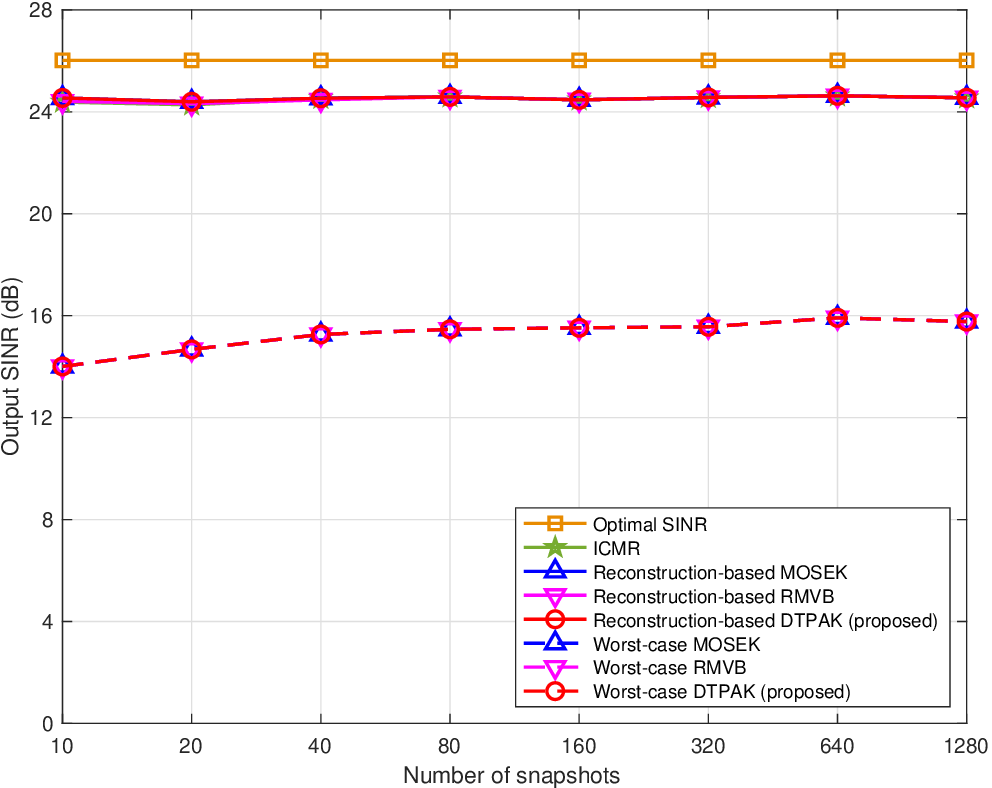}
\par\end{centering}
\caption{\label{fig:SINR vs sample-1}Output SINR versus snapshot number in
the case of signal look direction mismatch. $N=40$.}
\end{figure}

The algorithmic performance is assessed in terms of computational
time and the comparison is shown in Figure \ref{fig:time}. Under
the worst-case parameters, DTPAK is the most efficient, and the
running time is reduced by $75\%$ and $88\%$ compared to RMVB and
MOSEK. Under the reconstruction-based parameters, DTPAK is as fast as RMVB
and faster than solver-based methods, i.e., ICMR and MOSEK. The efficiency
improvement is about $45\%$. The formulation with reconstruction-based parameters
consumes more computational time for two reasons. Firstly, the computation
of the reconstruction-based parameters involves matrix inversion and factorization,
which increases time complexity. Secondly, the matrix $\mathbf{A}$
in the reconstruction-based parameters has fewer rows than columns, leading
to more complex operations in the Diagonalization Transform stage. 

\begin{figure}[tbh]
\begin{centering}
\includegraphics[width=0.8\columnwidth]{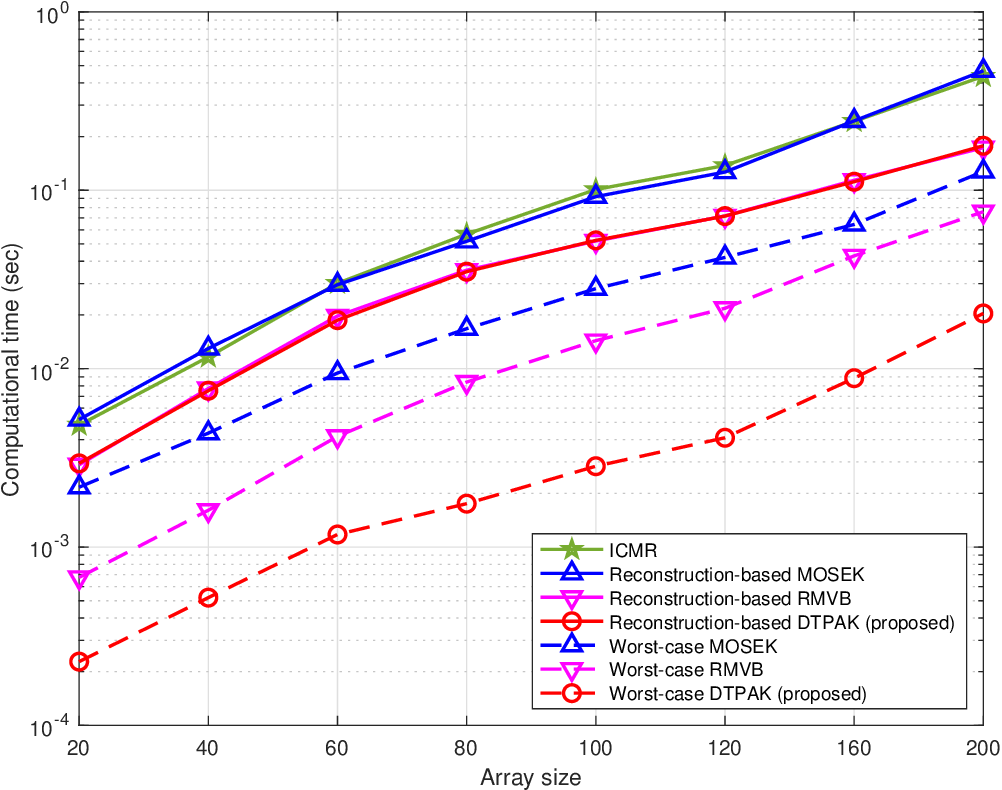}
\par\end{centering}
\caption{\label{fig:time}Computation time in the case of signal look direction
mismatch.}
\end{figure}

\subsection{Gain and Phase Perturbation Mismatch}

In the second example, the true steering vector is influenced by gain
and phase perturbations in each antenna element. The gain and phase
perturbations follow Gaussian distributions $\mathcal{N}\left(1,0.03^{2}\right)$
and $\mathcal{N}\left(0,\left(0.02\pi\right)^{2}\right)$, respectively.
The true steering vector remains fixed within each snapshot block
but varies across different instances. 

The comparison of output SINR is presented in Figures \ref{fig:SINR vs SNR-2}
and \ref{fig:SINR vs sample-2}. The ICMR beamformer can still be
perfectly reproduced regardless of gain and phase perturbations, but
there is a non-negligible gap to the optimal SINR, around $7$ dB.
This is because the theoretical steering vector expression is inaccurate
for an uncalibrated array and the estimation error of $\hat{\mathbf{R}}_{i+n}$
is relatively large. Despite suffering from self-cancellation of the
desired signal at high input SNR, the worst-case-constrained beamforming
is superior in the low input SNR regime and thus remains of particular
interest in the presence of calibration errors. The computational
behaviors of the compared algorithms are very similar to those in
Figure \ref{fig:time} because the new steering vector mismatch model does not change the algorithm architecture. 

\begin{figure}[tbh]
\begin{centering}
\includegraphics[width=0.8\columnwidth]{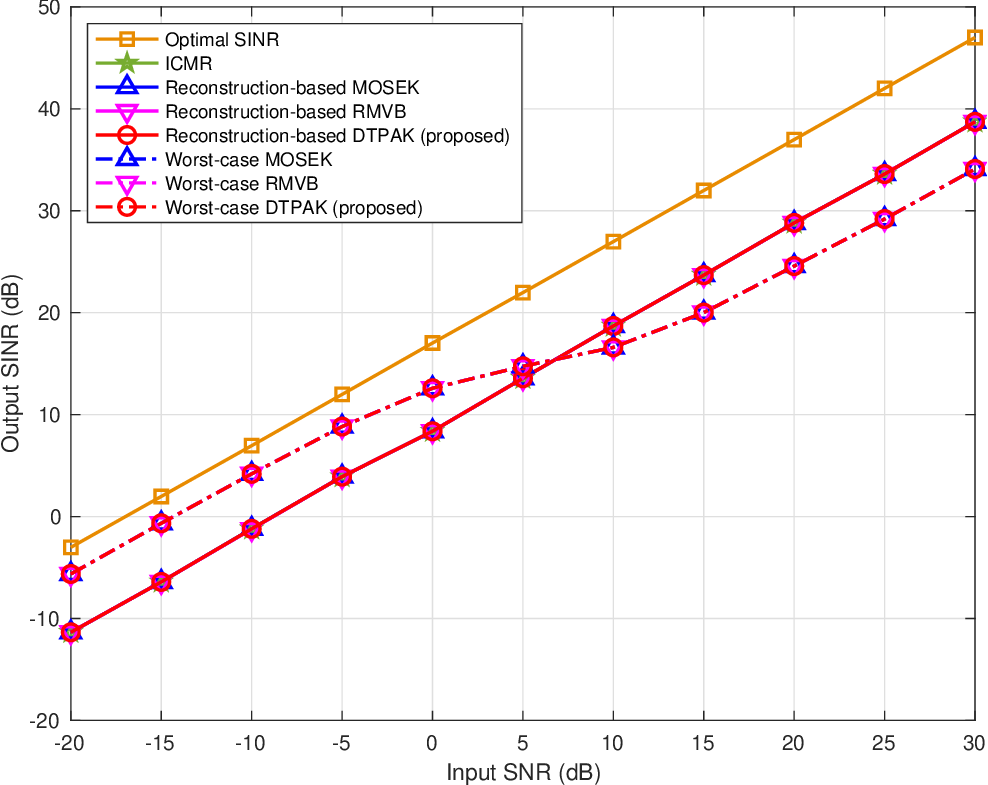}
\par\end{centering}
\caption{\label{fig:SINR vs SNR-2}Output SINR versus input SNR in the case
of gain and phase perturbation mismatch.}
\end{figure}

\begin{figure}[tbh]
\begin{centering}
\includegraphics[width=0.8\columnwidth]{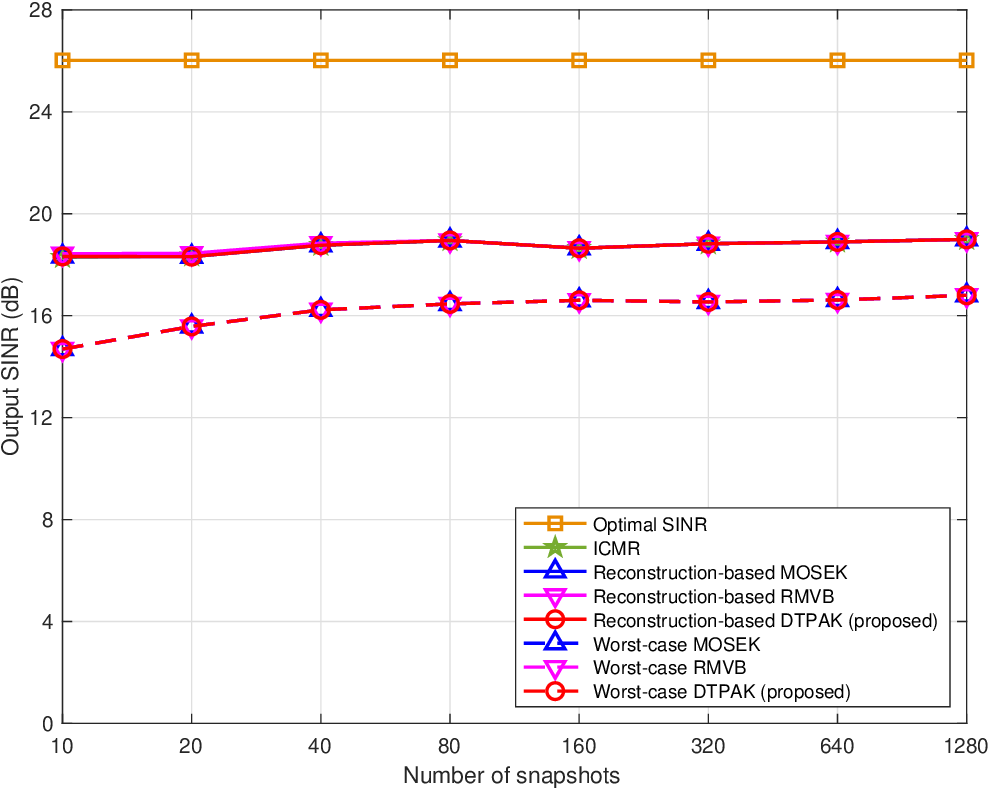}
\par\end{centering}
\caption{\label{fig:SINR vs sample-2}Output SINR versus snapshot number in
the case of gain and phase perturbation mismatch. $N=40$. }
\end{figure}

\subsection{Incoherent Local Scattering Mismatch}

The third example illustrates the impact of incoherent local scattering
on the steering vector mismatch. The desired signal has a time-varying
steering vector and is modeled as \citep{besson2000decoupled}
\begin{equation}
\mathbf{a}\left[t\right]=\varphi_{0}\left[t\right]\cdot\mathbf{a}\left(\theta_{s}\right)+\sum_{p=1}^{3}\varphi_{p}\left[t\right]\cdot\mathbf{a}\left(\theta_{p}\right)
\end{equation}
where $\varphi_{0}\left[t\right]$ and $\varphi_{p}\left[t\right]$,
$p=1,2,3$, are independently and identically distributed (i.i.d.)
standard complex Gaussian random variables. The DOAs $\theta_{p}$
are statistically independent and Gaussian distributed with mean $\theta_{s}$
and standard deviation $1^{\lyxmathsym{\textdegree}}$. The $\theta_{p}$
values stay fixed in each snapshot block but vary across instances.
The random variables $\varphi_{0}\left[t\right]$ and $\varphi_{p}\left[t\right]$
change within each snapshot block and also from instance to instance.
In this example, the signal covariance matrix $\mathbf{R}_{s}$ is
no longer rank-one and the output SINR is modified as \citep{gershman1999robust}
\begin{equation}
\mathrm{SINR}=\frac{\mathbf{w}^{\H}\mathbf{R}_{s}\mathbf{w}}{\mathbf{w}^{\H}\mathbf{R}_{i+n}\mathbf{w}}.\label{eq:ratio}
\end{equation}
The maximizer of the ratio \eqref{eq:ratio} is given by $\mathcal{P}\left\{ \mathbf{R}_{i+n}^{-1}\mathbf{R}_{s}\right\} $
where $\mathcal{P}\left\{ \cdot\right\} $ is the principal eigenvector
operator \citep{gershman1999robust}. 

The comparison of output SINR is presented in Figures \ref{fig:SINR vs SNR-3}
and \ref{fig:SINR vs sample-3}. The ICMR beamformer
is still recovered and the optimality gap is small. The unified formulation
can produce both types of beamformer, simply by switching the problem
parameters. The reconstruction-based beamformer yields higher output
SINR than the worst-case-constrained beamformer throughout the range
of input SNR. The asymptotic improvement is 10 dB. The unified formulation
is convex, so different algorithms can all converge to a global optimum.
The running time comparison is largely consistent with the results
shown in Figure \ref{fig:time}. 

\begin{figure}[tbh]
\begin{centering}
\includegraphics[width=0.8\columnwidth]{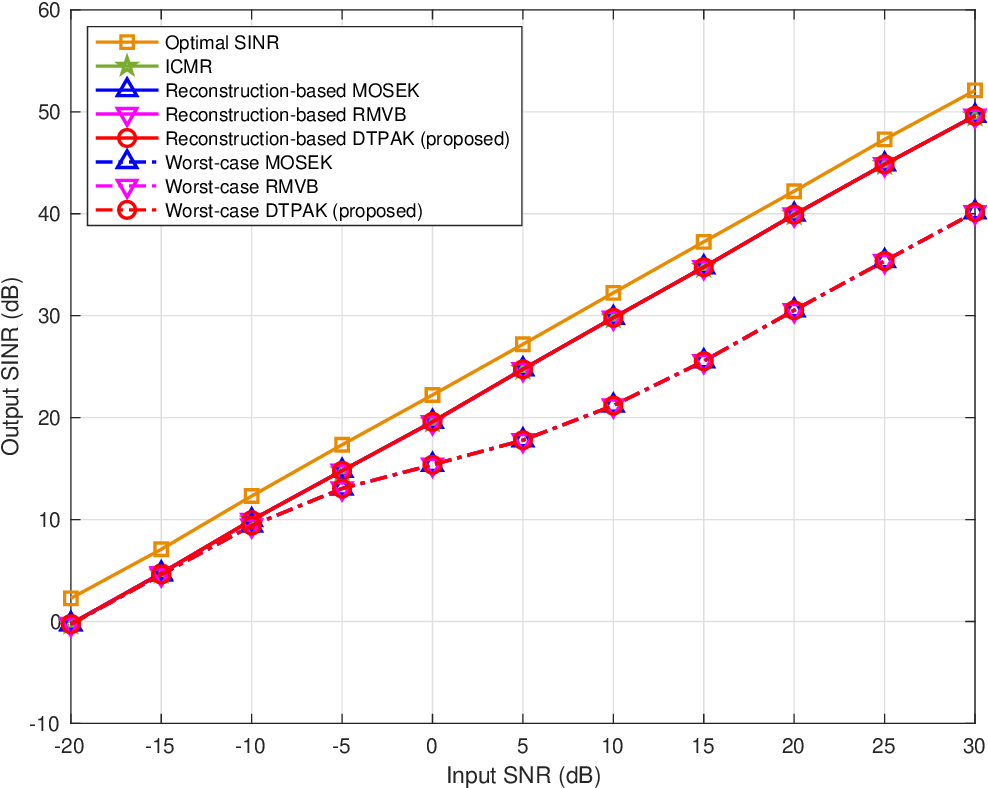}
\par\end{centering}
\caption{\label{fig:SINR vs SNR-3}Output SINR versus input SNR in the case
of incoherent local scattering mismatch.}
\end{figure}

\begin{figure}[tbh]
\begin{centering}
\includegraphics[width=0.8\columnwidth]{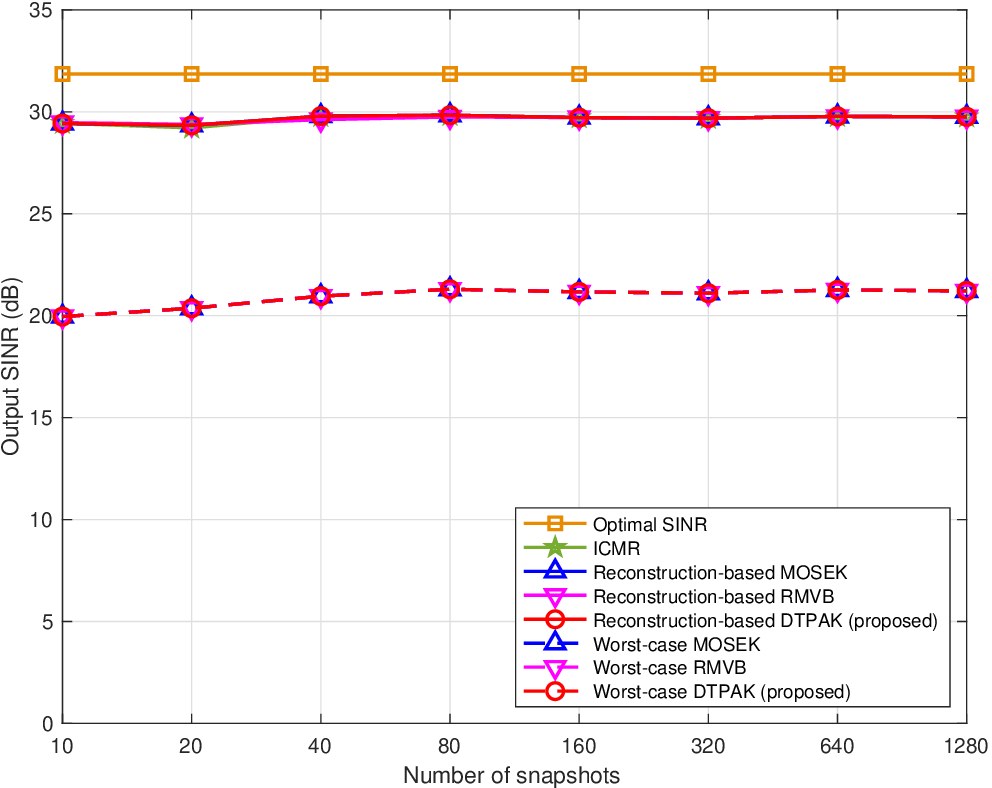}
\par\end{centering}
\caption{\label{fig:SINR vs sample-3}Output SINR versus snapshot number in
the case of incoherent local scattering mismatch. $N=40$.}
\end{figure}

\section{Conclusion\label{sec:Conclusion}}

We have revisited the robust adaptive beamforming problem and shown that its two
mainstream formulations are not alternatives. Theorem \ref{thm:equivalence}
establishes a primal-dual equivalence between the reconstruction-based and the
worst-case-constrained approaches, which therefore share a single 
problem formulation under different choices of its parameters. The desired signal
cancellation deficiency long attributed to the worst-case-constrained beamformer is, in
consequence, a property of the parameters with which that formulation is
instantiated rather than of the formulation itself. It is removed by the
parameter choice the reconstruction-based approach supplies.

For the unified problem we have derived a closed-form solution
in three stages: Diagonalization
Transform, Phase Alignment, and KKT Solution. It is named DTPAK and elaborated
in Theorem \ref{thm:unified}. In place of the
Lagrange-multiplier equation that RMVB solves iteratively, an auxiliary variable
fusing the primal variable with the multiplier reduces the KKT system to a single
scalar equation solvable by bisection. Every remaining variable is then
recovered from it. The scheme calls no optimization solver, keeps the problem in the
complex domain rather than doubling its size, and covers the rank-deficient
regimes that RMVB excludes.

Because the solution is available in closed form, questions that a solver cannot
answer become tractable. Theorem \ref{thm:unified}
gives necessary and sufficient conditions for the existence and the uniqueness of
the optimal solution, neither of which had been characterized before. Two of
their consequences deserve emphasis. An optimal solution may fail to exist even
when the feasible set is non-empty, a case in which a numerical solver returns an
approximately optimal point with no indication that there is nothing to
approximate. And the solution may be unique even when the objective is not
strictly convex, as the constraint supplies the curvature that a rank-deficient
covariance matrix lacks. Together these conditions delimit the values of the
uncertainty parameter for which the problem is well posed, a question that must
otherwise be settled by trial.

The numerical results confirm the equivalence by recovering the classical ICMR
beamformer from the unified formulation, and show that DTPAK attains the same
optimum as a general-purpose solver at a lower computational cost.

{\small

\bibliographystyle{IEEEtranN}
\bibliography{IEEEabrv,LichengZhao}

}

\clearpage

\setcounter{page}{1}
\begin{appendices}

{\twocolumn[\centering \Huge{Supplementary Material to ``Worst-Case and
Covariance-Reconstruction Robust Adaptive Beamforming: Equivalence and Closed-Form Solution''}]}

\section{\label{sec:Proof-of-Lemma transform}Proof of Lemma \ref{lem:transform}}
\begin{IEEEproof}
Consider the quadratic constraint first. Expanding it and cancelling the common
term $\mathbf{a}^{\H}\hat{\mathbf{R}}_{i+n}\mathbf{a}$, substituting
$\bm{\delta}_{\bot}=\mathbf{S}_{a}\mathbf{h}$, and introducing
$\bar{\mathbf{R}}$ and $\bar{\mathbf{a}}$ gives
\begin{equation}
\begin{aligned} & \left(\mathbf{a}+\bm{\delta}_{\bot}\right)^{\H}\hat{\mathbf{R}}_{i+n}\left(\mathbf{a}+\bm{\delta}_{\bot}\right)\leq\mathbf{a}^{\H}\hat{\mathbf{R}}_{i+n}\mathbf{a}\\
\Longleftrightarrow & \bm{\delta}_{\bot}^{\H}\hat{\mathbf{R}}_{i+n}\bm{\delta}_{\bot}+2\mathrm{Re}\left[\mathbf{a}^{\H}\hat{\mathbf{R}}_{i+n}\bm{\delta}_{\bot}\right]\leq0\\
\Longleftrightarrow & \mathbf{h}^{\H}\mathbf{S}_{a}^{\H}\hat{\mathbf{R}}_{i+n}\mathbf{S}_{a}\mathbf{h}+2\mathrm{Re}\left[\mathbf{a}^{\H}\hat{\mathbf{R}}_{i+n}\mathbf{S}_{a}\mathbf{h}\right]\leq0\\
\Longleftrightarrow & \mathbf{h}^{\H}\bar{\mathbf{R}}\mathbf{h}+2\mathrm{Re}\left[\bar{\mathbf{a}}^{\H}\mathbf{h}\right]+\bar{\mathbf{a}}^{\H}\bar{\mathbf{R}}^{-1}\bar{\mathbf{a}}\leq\bar{\mathbf{a}}^{\H}\bar{\mathbf{R}}^{-1}\bar{\mathbf{a}}\\
\Longleftrightarrow & \left\Vert \bar{\mathbf{R}}^{1/2}\mathbf{h}+\bar{\mathbf{R}}^{-1/2}\bar{\mathbf{a}}\right\Vert _{2}\leq\sqrt{\bar{\mathbf{a}}^{\H}\bar{\mathbf{R}}^{-1}\bar{\mathbf{a}}},
\end{aligned}
\end{equation}
where the fourth step adds $\bar{\mathbf{a}}^{\H}\bar{\mathbf{R}}^{-1}\bar{\mathbf{a}}$
to both sides so that the left-hand side becomes a perfect square, and the fifth
takes square roots. This is \eqref{eq:transform-constraint}, which becomes
\eqref{eq:transform-ball} once $\mathbf{q}$ and $\varepsilon$ are read from
\eqref{eq:transform-defs}.

For the objective, solving the definition of $\mathbf{q}$ for $\mathbf{h}$ gives
$\mathbf{h}=\bar{\mathbf{R}}^{-1/2}\left(\mathbf{q}-\bar{\mathbf{R}}^{-1/2}\bar{\mathbf{a}}\right)$,
so that the corrected steering vector is
\begin{equation}
\begin{aligned}\mathbf{a}+\bm{\delta}_{\bot} & =\mathbf{a}+\mathbf{S}_{a}\mathbf{h}\\
 & =\mathbf{a}-\mathbf{S}_{a}\bar{\mathbf{R}}^{-1}\bar{\mathbf{a}}+\mathbf{S}_{a}\bar{\mathbf{R}}^{-1/2}\mathbf{q}\\
 & =\mathbf{f}+\mathbf{F}\mathbf{q}.
\end{aligned}
\end{equation}
Substituting this into the objective of \eqref{eq: icmr} yields
\eqref{eq:transform-objective}. Note that the change of variables is a bijection,
since $\bar{\mathbf{R}}\succ\mathbf{0}$ makes $\bar{\mathbf{R}}^{1/2}$ invertible,
so the transformed problem is equivalent to the original rather than merely a
relaxation of it.
\end{IEEEproof}

\section{\label{sec:Proof-of-Lemma qcqp2socp}Proof of Lemma \ref{lem:qcqp2socp}}
\begin{IEEEproof}
Every $\mathbf{w}\in\mathbb{C}^{N}$ can be written as
$\mathbf{w}=\beta\bar{\mathbf{w}}$ with $\beta\geq0$, so the dual problem may be
reparameterized in $\left(\beta,\bar{\mathbf{w}}\right)$ without loss of
optimality. Collecting powers of $\beta$, the objective becomes
\begin{equation}
\begin{aligned}g\left(\beta,\bar{\mathbf{w}}\right) & =-\frac{1}{4}\left(\beta\bar{\mathbf{w}}\right)^{\H}\hat{\mathbf{R}}_{i+n}\left(\beta\bar{\mathbf{w}}\right)\\
 & \quad-\varepsilon\left\Vert \mathbf{F}^{\H}\left(\beta\bar{\mathbf{w}}\right)\right\Vert _{2}+\mathrm{Re}\left[\mathbf{\left(\beta\bar{\mathbf{w}}\right)}^{\H}\mathbf{f}\right]\\
 & =-\frac{1}{4}\left(\bar{\mathbf{w}}^{\H}\hat{\mathbf{R}}_{i+n}\bar{\mathbf{w}}\right)\beta^{2}+c\left(\bar{\mathbf{w}}\right)\beta,
\end{aligned}
\label{eq:polar-objective}
\end{equation}
where $c\left(\bar{\mathbf{w}}\right)\triangleq\mathrm{Re}\left[\bar{\mathbf{w}}^{\H}\mathbf{f}\right]-\varepsilon\left\Vert \mathbf{F}^{\H}\bar{\mathbf{w}}\right\Vert _{2}$.
For fixed $\bar{\mathbf{w}}\neq\mathbf{0}$ this is a strictly concave quadratic
in $\beta$, since $\hat{\mathbf{R}}_{i+n}\succ\mathbf{0}$, and its maximizer over
$\beta\geq0$ is
\begin{equation}
\beta^{\star}=\left[\frac{2c\left(\bar{\mathbf{w}}\right)}{\bar{\mathbf{w}}^{\H}\hat{\mathbf{R}}_{i+n}\bar{\mathbf{w}}}\right]_{+},\qquad\left[x\right]_{+}=\max\left(x,0\right),
\end{equation}
giving
$\max_{\beta\geq0}g=c\left(\bar{\mathbf{w}}\right)^{2}/\bar{\mathbf{w}}^{\H}\hat{\mathbf{R}}_{i+n}\bar{\mathbf{w}}$
when $c\left(\bar{\mathbf{w}}\right)>0$, and $\max_{\beta\geq0}g=0$ otherwise.
The same value $0$ is attained at $\bar{\mathbf{w}}=\mathbf{0}$.

It remains to rule out those degenerate points, which requires exhibiting a
direction with $c\left(\bar{\mathbf{w}}\right)>0$. The nominal steering vector
itself is such a direction. Since $\mathbf{S}_{a}^{\H}\mathbf{a}=\mathbf{0}$, we have
$\mathbf{F}^{\H}\mathbf{a}=\bar{\mathbf{R}}^{-1/2}\mathbf{S}_{a}^{\H}\mathbf{a}=\mathbf{0}$
and
$\mathrm{Re}\left[\mathbf{a}^{\H}\mathbf{f}\right]=\mathrm{Re}\left[\mathbf{a}^{\H}\mathbf{a}-\mathbf{a}^{\H}\mathbf{S}_{a}\bar{\mathbf{R}}^{-1}\bar{\mathbf{a}}\right]=\mathbf{a}^{\H}\mathbf{a}$,
so that
\begin{equation}
c\left(\mathbf{a}\right)=\mathbf{a}^{\H}\mathbf{a}=\left\Vert \mathbf{a}\right\Vert _{2}^{2}>0,
\end{equation}
the inequality being strict because a steering vector is nonzero. Evaluating
\eqref{eq:polar-objective} at $\bar{\mathbf{w}}=\mathbf{a}$ with the
corresponding $\beta^{\star}$ gives the dual value
$\left(\mathbf{a}^{\H}\mathbf{a}\right)^{2}/\mathbf{a}^{\H}\hat{\mathbf{R}}_{i+n}\mathbf{a}>0$.
The optimal value of the dual is therefore strictly positive, and since both
$\bar{\mathbf{w}}=\mathbf{0}$ and every $\bar{\mathbf{w}}$ with
$c\left(\bar{\mathbf{w}}\right)\leq0$ attain only $0$, the optimal
$\bar{\mathbf{w}}$ lies in neither. The set
$\left\{ \bar{\mathbf{w}}\,|\,c\left(\bar{\mathbf{w}}\right)>0\right\}$ is thus
nonempty and contains the optimum, so the projection $\left[\cdot\right]_{+}$ in
$\beta^{\star}$ may be dropped and the dual problem written as
\begin{equation}
\begin{aligned} & \underset{\bar{\mathbf{w}}\in\mathbb{C}^{N}}{\mathsf{maximize}} &  & \frac{\left(\mathrm{Re}\left[\bar{\mathbf{w}}^{\H}\mathbf{f}\right]-\varepsilon\left\Vert \mathbf{F}^{\H}\bar{\mathbf{w}}\right\Vert _{2}\right)^{2}}{\bar{\mathbf{w}}^{\H}\hat{\mathbf{R}}_{i+n}\bar{\mathbf{w}}}\\
 & \mathsf{subject\thinspace to} &  & \mathrm{Re}\left[\bar{\mathbf{w}}^{\H}\mathbf{f}\right]-\varepsilon\left\Vert \mathbf{F}^{\H}\bar{\mathbf{w}}\right\Vert _{2}>0,
\end{aligned}
\end{equation}
whose constraint also excludes $\bar{\mathbf{w}}=\mathbf{0}$. Numerator and
denominator are both second-order homogeneous, so the objective is invariant to
the scale of $\bar{\mathbf{w}}$ and the numerator may be normalized to $1$
without loss of optimality. Maximizing the resulting objective is then the same
as minimizing its denominator, which yields the stated formulation:
\begin{equation}
\begin{aligned} & \underset{\bar{\mathbf{w}}\in\mathbb{C}^{N}}{\mathsf{minimize}} &  & \bar{\mathbf{w}}^{\H}\hat{\mathbf{R}}_{i+n}\bar{\mathbf{w}}\\
 & \mathsf{subject\thinspace to} &  & \mathrm{Re}\left[\bar{\mathbf{w}}^{\H}\mathbf{f}\right]-\varepsilon\left\Vert \mathbf{F}^{\H}\bar{\mathbf{w}}\right\Vert _{2}=1.
\end{aligned}
\end{equation}

Finally, this problem has the same optimal value and the same optimal solutions as
\eqref{eq:reduced-dual}, that is, as \eqref{eq:opt prob} under \eqref{eq:identification}.
Write $c\left(\bar{\mathbf{w}}\right)$ as above and $D\left(\bar{\mathbf{w}}\right)=\bar{\mathbf{w}}^{\H}\hat{\mathbf{R}}_{i+n}\bar{\mathbf{w}}$.
If $c\left(\bar{\mathbf{w}}\right)>1$, the point $\bar{\mathbf{w}}/c\left(\bar{\mathbf{w}}\right)$
has $c=1$ and a strictly smaller $D$, so at any optimum of either problem the
inequality constraint is active. Next, for any $\bar{\mathbf{w}}$ let $\bar{\mathbf{w}}'=e^{j\phi}\bar{\mathbf{w}}$
with $\phi$ chosen so that $\bar{\mathbf{w}}'^{\H}\mathbf{f}\geq0$. Then $D$ and
$\left\Vert \mathbf{F}^{\H}\bar{\mathbf{w}}\right\Vert _{2}$ are unchanged and
$\mathrm{Re}\left[\bar{\mathbf{w}}'^{\H}\mathbf{f}\right]=\left|\bar{\mathbf{w}}^{\H}\mathbf{f}\right|\geq\mathrm{Re}\left[\bar{\mathbf{w}}^{\H}\mathbf{f}\right]$,
with equality if and only if $\mathrm{Im}\left[\bar{\mathbf{w}}^{\H}\mathbf{f}\right]=0$
and $\mathrm{Re}\left[\bar{\mathbf{w}}^{\H}\mathbf{f}\right]\geq0$. An optimal $\bar{\mathbf{w}}$ of
the equality-constrained problem has $c\left(\bar{\mathbf{w}}\right)=1$, and if the
inequality were strict, $\bar{\mathbf{w}}'/c\left(\bar{\mathbf{w}}'\right)$ would be feasible with a
strictly smaller $D$; hence it satisfies $\mathrm{Im}\left[\bar{\mathbf{w}}^{\H}\mathbf{f}\right]=0$
and is feasible for \eqref{eq:reduced-dual}. Conversely, a point feasible for
\eqref{eq:reduced-dual} has $c\geq1$, so $\bar{\mathbf{w}}/c$ is feasible for the
equality-constrained problem with no larger $D$. The two optimal values therefore
coincide, and an optimal point of one is optimal for the other.
\end{IEEEproof}

\section{\label{sec:Proof-of-Theorem equivalence}Proof of Theorem \ref{thm:equivalence}}
\begin{IEEEproof}
The first claim collects the steps of Section \ref{sec:Connection-data-model}.
Lemma \ref{lem:transform} recasts \eqref{eq: icmr} as the minimization of
$\left(\mathbf{f}+\mathbf{F}\mathbf{q}\right)^{\H}\hat{\mathbf{R}}_{i+n}^{-1}\left(\mathbf{f}+\mathbf{F}\mathbf{q}\right)$
over $\left\Vert \mathbf{q}\right\Vert _{2}\leq\varepsilon$, whose dual is
\eqref{eq:data-driven dual}; Lemma \ref{lem:qcqp2socp} and the two reductions
following it bring that dual to \eqref{eq:reduced-dual}, which is
\eqref{eq:opt prob} under \eqref{eq:identification}.

For the second claim, let $\bm{\delta}_{\bot}^{\star}$ be optimal for
\eqref{eq: icmr}, $\mathbf{w}^{\star}$ optimal for the dual
\eqref{eq:data-driven dual}, and $\bar{\mathbf{w}}^{\star}$ optimal for
\eqref{eq:opt prob} under \eqref{eq:identification}, at which the constraint is
active. Write
$c\left(\mathbf{w}\right)\triangleq\mathrm{Re}\left[\mathbf{w}^{\H}\mathbf{f}\right]-\varepsilon\left\Vert \mathbf{F}^{\H}\mathbf{w}\right\Vert _{2}$
as in Appendix \ref{sec:Proof-of-Lemma qcqp2socp}, so that $c\left(\bar{\mathbf{w}}^{\star}\right)=1$
and the dual objective is $g\left(\mathbf{w}\right)=-\frac{1}{4}\mathbf{w}^{\H}\hat{\mathbf{R}}_{i+n}\mathbf{w}+c\left(\mathbf{w}\right)$.
Three facts are combined.

First, strong duality holds between \eqref{eq:icmr-split} and
\eqref{eq:data-driven dual}: the problem is convex, and $\mathbf{q}=\mathbf{0}$
with $\mathbf{z}=\hat{\mathbf{R}}_{i+n}^{-1/2}\mathbf{f}$ satisfies the equality
constraint with $\left\Vert \mathbf{q}\right\Vert _{2}<\varepsilon$ (here $\varepsilon>0$ is used), so Slater's
condition is met. Strong duality therefore holds, and the dual optimum $\mathbf{w}^{\star}$ is attained. Hence
\begin{equation}
\left(\mathbf{a}+\bm{\delta}_{\bot}^{\star}\right)^{\H}\hat{\mathbf{R}}_{i+n}^{-1}\left(\mathbf{a}+\bm{\delta}_{\bot}^{\star}\right)=g\left(\mathbf{w}^{\star}\right).\label{eq:strong-duality}
\end{equation}
Second, the stationarity condition
$\mathbf{z}=\frac{1}{2}\hat{\mathbf{R}}_{i+n}^{1/2}\mathbf{w}^{\star}$ together
with the equality constraint of \eqref{eq:icmr-split} gives
\begin{equation}
\mathbf{w}^{\star}=2\hat{\mathbf{R}}_{i+n}^{-1}\left(\mathbf{a}+\bm{\delta}_{\bot}^{\star}\right).\label{eq:w-star-primal}
\end{equation}
Third, Lemma \ref{lem:qcqp2socp} gives
$\mathbf{w}^{\star}=\beta^{\star}\bar{\mathbf{w}}^{\star}$ with, since
$c\left(\bar{\mathbf{w}}^{\star}\right)=1$,
\begin{equation}
\beta^{\star}=\frac{2}{\left(\bar{\mathbf{w}}^{\star}\right)^{\H}\hat{\mathbf{R}}_{i+n}\bar{\mathbf{w}}^{\star}}.\label{eq:beta-star-normalized}
\end{equation}

Now evaluate the right-hand side of \eqref{eq:strong-duality}. Because $c$ is
positively homogeneous of degree one,
$c\left(\beta^{\star}\bar{\mathbf{w}}^{\star}\right)=\beta^{\star}c\left(\bar{\mathbf{w}}^{\star}\right)=\beta^{\star}$,
and with $D\triangleq\left(\bar{\mathbf{w}}^{\star}\right)^{\H}\hat{\mathbf{R}}_{i+n}\bar{\mathbf{w}}^{\star}$,
\begin{equation}
g\left(\mathbf{w}^{\star}\right)=-\frac{1}{4}D\left(\beta^{\star}\right)^{2}+\beta^{\star}=-\frac{1}{D}+\frac{2}{D}=\frac{1}{D}.
\end{equation}
Thus \eqref{eq:strong-duality} reads
$D=1/\left(\mathbf{a}+\bm{\delta}_{\bot}^{\star}\right)^{\H}\hat{\mathbf{R}}_{i+n}^{-1}\left(\mathbf{a}+\bm{\delta}_{\bot}^{\star}\right)$,
which says that the optimal value $D$ of \eqref{eq:opt prob} is the reciprocal of the optimal value of \eqref{eq: icmr}. Combining \eqref{eq:w-star-primal} with \eqref{eq:beta-star-normalized},
\begin{equation}
\begin{aligned}\bar{\mathbf{w}}^{\star}=\frac{\mathbf{w}^{\star}}{\beta^{\star}} & =D\,\hat{\mathbf{R}}_{i+n}^{-1}\left(\mathbf{a}+\bm{\delta}_{\bot}^{\star}\right)\\
 & =\frac{\hat{\mathbf{R}}_{i+n}^{-1}\left(\mathbf{a}+\bm{\delta}_{\bot}^{\star}\right)}{\left(\mathbf{a}+\bm{\delta}_{\bot}^{\star}\right)^{\H}\hat{\mathbf{R}}_{i+n}^{-1}\left(\mathbf{a}+\bm{\delta}_{\bot}^{\star}\right)},
\end{aligned}
\end{equation}
which is \eqref{eq:icmr-beamformer}. The normalization in \eqref{eq:opt prob} is thus
the scale at which the worst-case solution coincides with the
reconstruction-based beamformer.
\end{IEEEproof}

\section{\label{sec:Proof-of-Lemma reduced data}Proof of Lemma \ref{lem:reduced data}}
\begin{IEEEproof}
We verify that the matrices constructed in Section \ref{sec:Problem and Solution}-A work.

\emph{Case $r=N$.} With $\mathbf{T}=\mathbf{B}^{-1}\mathbf{U}$, where $\mathbf{B}^{\H}\mathbf{B}=\mathbf{A}^{\H}\mathbf{A}$ and
$\left(\mathbf{B}^{-1}\right)^{\H}\mathbf{R}\mathbf{B}^{-1}=\mathbf{U}\bm{\Lambda}\mathbf{U}^{\H}$,
\begin{equation}
\begin{aligned}\mathbf{T}^{\H}\mathbf{A}^{\H}\mathbf{A}\mathbf{T} & =\mathbf{U}^{\H}\left(\mathbf{B}^{-1}\right)^{\H}\mathbf{B}^{\H}\mathbf{B}\mathbf{B}^{-1}\mathbf{U}=\mathbf{I},\\
\mathbf{T}^{\H}\mathbf{R}\mathbf{T} & =\mathbf{U}^{\H}\left(\mathbf{B}^{-1}\right)^{\H}\mathbf{R}\mathbf{B}^{-1}\mathbf{U}=\bm{\Lambda}\succeq\mathbf{0}.
\end{aligned}
\end{equation}

\emph{Case $r<N$.} Write $\mathbf{T}=\tilde{\mathbf{B}}\mathbf{T}_{1}$, where $\mathbf{T}_{1}$ is the block matrix in the
definition of $\mathbf{T}$. The matrix $\tilde{\mathbf{B}}=\left[\mathbf{V}_{1}\left[\mathrm{Diag}\left(\mathbf{d}\right)\right]^{-1},\mathbf{V}_{2}\right]$
is invertible, being the product of the unitary $\left[\mathbf{V}_{1},\mathbf{V}_{2}\right]$ and
$\mathrm{BlkDiag}\left(\left\{ \left[\mathrm{Diag}\left(\mathbf{d}\right)\right]^{-1},\mathbf{I}\right\} \right)$.
Because $\mathbf{V}_{1}^{\H}\mathbf{V}_{1}=\mathbf{I}$ and $\mathbf{V}_{1}^{\H}\mathbf{V}_{2}=\mathbf{0}$,
$\mathbf{A}\tilde{\mathbf{B}}=\left[\mathbf{U}_{A},\mathbf{0}\right]$ with
$\mathbf{U}_{A}^{\H}\mathbf{U}_{A}=\mathbf{I}$, whence
$\tilde{\mathbf{B}}^{\H}\mathbf{A}^{\H}\mathbf{A}\tilde{\mathbf{B}}=\mathrm{BlkDiag}\left(\{\mathbf{I}_{r},\mathbf{0}\}\right)$.
The first block row of $\mathbf{T}_{1}$ is $\left[\mathbf{U}_{1},\mathbf{0}\right]$, so
\begin{equation}
\begin{aligned}\mathbf{T}_{1}^{\H}\,\mathrm{BlkDiag}\left(\{\mathbf{I}_{r},\mathbf{0}\}\right)\mathbf{T}_{1} & =\mathrm{BlkDiag}\left(\{\mathbf{U}_{1}^{\H}\mathbf{U}_{1},\mathbf{0}\}\right)\\
 & =\mathrm{BlkDiag}\left(\{\mathbf{I}_{r},\mathbf{0}\}\right),
\end{aligned}
\end{equation}
which gives the first identity of the lemma.

For the second, $\tilde{\mathbf{R}}=\tilde{\mathbf{B}}^{\H}\mathbf{R}\tilde{\mathbf{B}}\succ\mathbf{0}$, so
$\tilde{\mathbf{R}}_{22}\succ\mathbf{0}$ is invertible and the Schur complement is well defined. The lower-left block of
$\mathbf{T}_{1}$ is chosen so that the $(2,1)$ block of $\tilde{\mathbf{R}}\mathbf{T}_{1}$ vanishes,
\begin{equation}
\tilde{\mathbf{R}}\mathbf{T}_{1}=\left[\begin{array}{cc}
\tilde{\mathbf{S}}_{11}\mathbf{U}_{1} & \tilde{\mathbf{R}}_{12}\mathbf{U}_{2}\\
\mathbf{0} & \tilde{\mathbf{R}}_{22}\mathbf{U}_{2}
\end{array}\right],
\end{equation}
$\tilde{\mathbf{S}}_{11}$ being the Schur complement by construction. Multiplying on the left by
$\mathbf{T}_{1}^{\H}$ cancels the remaining off-diagonal block, since
$\tilde{\mathbf{R}}_{21}^{\H}=\tilde{\mathbf{R}}_{12}$, and the two EVDs diagonalize what is left:
\begin{equation}
\mathbf{T}_{1}^{\H}\tilde{\mathbf{R}}\mathbf{T}_{1}=\left[\begin{array}{cc}
\mathbf{U}_{1}^{\H}\tilde{\mathbf{S}}_{11}\mathbf{U}_{1} & \mathbf{0}\\
\mathbf{0} & \mathbf{U}_{2}^{\H}\tilde{\mathbf{R}}_{22}\mathbf{U}_{2}
\end{array}\right]=\bm{\Lambda}.
\end{equation}
Hence $\mathbf{T}^{\H}\mathbf{R}\mathbf{T}=\mathbf{T}_{1}^{\H}\tilde{\mathbf{R}}\mathbf{T}_{1}=\mathrm{Diag}\left(\bm{\lambda}\right)$.
Finally $\mathbf{T}$ is invertible and $\mathbf{R}\succ\mathbf{0}$, so $\bm{\lambda}>\mathbf{0}$. 
\end{IEEEproof}

\section{\label{sec:Proof-of-Lemma phase align}Proof of Lemma \ref{lem:phase align}}
\begin{IEEEproof}
If $\mathbf{v}$ is feasible for \eqref{eq:opt prob v}, then $\mathbf{c}^{\T}\left|\mathbf{v}\right|\geq\mathrm{Re}\left[\mathbf{v}^{\H}\mathbf{b}\right]=\mathbf{v}^{\H}\mathbf{b}$,
so $\left|\mathbf{v}\right|$ is feasible for \eqref{eq:opt prob 3} with the same
objective value. Conversely, if $\mathbf{u}$ is feasible for \eqref{eq:opt prob 3},
then $\mathbf{v}=\left|\mathbf{u}\right|\odot\exp\left(j\arg\left(\mathbf{b}\right)\right)$
satisfies $\mathbf{v}^{\H}\mathbf{b}=\mathbf{c}^{\T}\left|\mathbf{u}\right|\geq\mathbf{c}^{\T}\mathbf{u}$,
which is real, and $\left\Vert \mathbf{v}_{1:r}\right\Vert _{2}=\left\Vert \mathbf{u}_{1:r}\right\Vert _{2}$,
so it is feasible for \eqref{eq:opt prob v} with the same objective value. Hence
the two problems have the same optimal value and one is solvable if and only if the other is.
Moreover, if $\mathbf{u}$ is optimal for \eqref{eq:opt prob 3}, so is $\left|\mathbf{u}\right|$ (it
is feasible with the same objective value), whence
$\left|\mathbf{u}\right|\odot\exp\left(j\arg\left(\mathbf{b}\right)\right)$ is
optimal for \eqref{eq:opt prob v}. For an index with $b_{n}=0$ we have $c_{n}=0$, and the
convention adopted for $\arg\left(b_{n}\right)$ changes neither $\mathbf{v}^{\H}\mathbf{b}$ nor
the objective.
\end{IEEEproof}

\section{\label{sec:Proof-of-Theorem unified}Proofs of Lemma \ref{lem:unique criterion} and Theorem \ref{thm:unified}}
\begin{lem}
\label{lem:exist eqv}The necessary and sufficient condition for the
existence of \textbf{$\mathbf{x}$ }such that $\mathbf{x}^{\T}\mathbf{y}\geq\varepsilon\left\Vert \mathbf{x}\right\Vert _{2}+1$
is $\left\Vert \mathbf{y}\right\Vert _{2}>\varepsilon$. 
\end{lem}
\begin{IEEEproof}
See Appendix \ref{sec:Proof-of-Lemma exist eqv} for the detailed
proof. 
\end{IEEEproof}
\begin{IEEEproof}[Proof of Lemma \ref{lem:unique criterion}]
\emph{Optimal value.} By Lemma \ref{lem:reduced data} and Lemma \ref{lem:phase align},
problem \eqref{eq:opt prob 3} has an optimal solution $\mathbf{u}$ with the same optimal
value $p^{\star}$. As recalled in Section \ref{subsec:KKT}, it is convex and satisfies
Slater's condition, so there is $\mu\geq0$ such that $\mu h\left(\mathbf{u}\right)=0$ and
\eqref{eq:KKT subgradient} holds. Multiplying \eqref{eq:KKT subgradient} by $u_{n}$ and
summing gives
\begin{equation}
\begin{aligned}2p^{\star}=2\sum_{n=1}^{N}\lambda_{n}u_{n}^{2} & =\mu\left(\mathbf{c}^{\T}\mathbf{u}-\varepsilon\,\mathbf{s}^{\T}\mathbf{u}_{1:r}\right)\\
 & =\mu\left(\mathbf{c}^{\T}\mathbf{u}-\varepsilon\left\Vert \mathbf{u}_{1:r}\right\Vert _{2}\right)\\
 & =\mu\left(h\left(\mathbf{u}\right)+1\right)=\mu,
\end{aligned}
\end{equation}
where $\mathbf{s}^{\T}\mathbf{u}_{1:r}=\left\Vert \mathbf{u}_{1:r}\right\Vert _{2}$ holds
both when $\mathbf{u}_{1:r}\neq\mathbf{0}$ (by the definition of $\mathbf{s}$) and when
$\mathbf{u}_{1:r}=\mathbf{0}$ (both sides vanish), and the last step uses
$\mu h\left(\mathbf{u}\right)=0$. This shows $p^{\star}=\mu/2$, which holds for every
multiplier, so $\mu$ is unique.

\emph{Uniqueness.} The statement concerns problem \eqref{eq:opt prob} itself, so no factorization
enters. Write $\psi\left(\mathbf{w}\right)\triangleq\mathrm{Re}\left[\mathbf{w}^{\H}\mathbf{a}\right]-\varepsilon\left\Vert \mathbf{A}\mathbf{w}\right\Vert _{2}-1$,
so that $\mathbf{w}$ is feasible if and only if $\psi\left(\mathbf{w}\right)\geq0$ and
$\mathrm{Im}\left[\mathbf{w}^{\H}\mathbf{a}\right]=0$. The feasible set is convex, as $\psi$
is concave. Moreover, for $\alpha>0$,
$\psi\left(\alpha\mathbf{w}\right)=\alpha\left(\psi\left(\mathbf{w}\right)+1\right)-1$.

\emph{Case $p^{\star}=0$.} Let $\mathbf{w}$ be optimal. Then $\mathbf{w}^{\H}\mathbf{R}\mathbf{w}=0$,
and $\mathbf{w}\neq\mathbf{0}$ because $\psi\left(\mathbf{0}\right)=-1$. For $t>1$,
$\psi\left(t\mathbf{w}\right)=t\left(\psi\left(\mathbf{w}\right)+1\right)-1>0$, so
$t\mathbf{w}$ is feasible with the same objective value $0$. Hence $t\mathbf{w}\neq\mathbf{w}$
is another optimal solution.

\emph{Case $p^{\star}>0$.} Let $\mathbf{w}_{1}$ and $\mathbf{w}_{2}$ be optimal. Their midpoint
$\mathbf{w}_{m}$ is feasible and, by the parallelogram identity,
\begin{equation}
\mathbf{w}_{m}^{\H}\mathbf{R}\mathbf{w}_{m}=p^{\star}-\tfrac{1}{4}\left(\mathbf{w}_{1}-\mathbf{w}_{2}\right)^{\H}\mathbf{R}\left(\mathbf{w}_{1}-\mathbf{w}_{2}\right)\leq p^{\star},
\end{equation}
so $\mathbf{w}_{m}$ is optimal as well and the inequality is an equality, that is,
$\mathbf{R}\left(\mathbf{w}_{1}-\mathbf{w}_{2}\right)=\mathbf{0}$. If $r<N$, then $\mathbf{R}\succ\mathbf{0}$
by assumption and $\mathbf{w}_{1}=\mathbf{w}_{2}$. If $r=N$, $\mathbf{A}$ is injective and
we argue with the constraint instead. Every optimal $\mathbf{w}$ has
$\psi\left(\mathbf{w}\right)=0$: otherwise $\alpha=1/\left(1+\psi\left(\mathbf{w}\right)\right)<1$
gives $\psi\left(\alpha\mathbf{w}\right)=0$, so $\alpha\mathbf{w}$ is feasible with objective
$\alpha^{2}p^{\star}<p^{\star}$. Applying this to $\mathbf{w}_{1}$, $\mathbf{w}_{2}$
and $\mathbf{w}_{m}$, and using that the real part of $\mathbf{w}^{\H}\mathbf{a}$ is linear,
\begin{equation}
\begin{aligned}0=\psi\left(\mathbf{w}_{m}\right)= & \tfrac{1}{2}\left(\psi\left(\mathbf{w}_{1}\right)+\psi\left(\mathbf{w}_{2}\right)\right)\\
 & +\varepsilon\left(\tfrac{1}{2}\left\Vert \mathbf{A}\mathbf{w}_{1}\right\Vert _{2}+\tfrac{1}{2}\left\Vert \mathbf{A}\mathbf{w}_{2}\right\Vert _{2}-\left\Vert \mathbf{A}\mathbf{w}_{m}\right\Vert _{2}\right),
\end{aligned}
\end{equation}
so $\left\Vert \mathbf{A}\mathbf{w}_{1}+\mathbf{A}\mathbf{w}_{2}\right\Vert _{2}=\left\Vert \mathbf{A}\mathbf{w}_{1}\right\Vert _{2}+\left\Vert \mathbf{A}\mathbf{w}_{2}\right\Vert _{2}$.
Both vectors are nonzero ($\mathbf{A}$ is injective and $\mathbf{w}_{i}\neq\mathbf{0}$), so
equality in the triangle inequality gives $\mathbf{A}\mathbf{w}_{2}=\theta\mathbf{A}\mathbf{w}_{1}$
with $\theta>0$, hence $\mathbf{w}_{2}=\theta\mathbf{w}_{1}$. Finally
$p^{\star}=\mathbf{w}_{2}^{\H}\mathbf{R}\mathbf{w}_{2}=\theta^{2}p^{\star}$ gives $\theta=1$.
\end{IEEEproof}
\begin{IEEEproof}[Proof of Theorem \ref{thm:unified}]
Throughout, $\mathbf{T}$, $\bm{\lambda}$, $\mathbf{b}$, $\mathbf{c}$, $\mathcal{I}_{\lambda,0}$,
$\mathcal{I}_{c,+}$, $\tau_{0}$ and $\tau_{r}$ are those of Section \ref{sec:Problem and Solution},
and $\mathbf{x}_{1:r}$ denotes the first $r$ entries of a vector $\mathbf{x}$.

\emph{Step 1: KKT conditions.} By Lemma \ref{lem:reduced data} and Lemma \ref{lem:phase align}, problem
\eqref{eq:opt prob} is solvable if and only if \eqref{eq:opt prob 3} is, with the same optimal value. As recalled in
Section \ref{subsec:KKT}, \eqref{eq:opt prob 3} is convex and satisfies Slater's condition whenever it is feasible, so
$\mathbf{u}$ is optimal if and only if there is $\mu\geq0$ such that
$\mu h\left(\mathbf{u}\right)=0$, $h\left(\mathbf{u}\right)\geq0$ and \eqref{eq:KKT subgradient} holds. Problem
\eqref{eq:opt prob 3} is therefore solvable if and only if a KKT point exists. Every KKT point falls in exactly one
of the following three cases.

\emph{Case A: $\mu=0$ and $\mathbf{u}_{1:r}\neq\mathbf{0}$.} Then \eqref{eq:KKT subgradient}
gives $u_{n}=0$ whenever $\lambda_{n}>0$, so $\mathbf{u}$ is supported on
$\mathcal{I}_{\lambda,0}$, which is nonempty only if $r=N$. Feasibility reads
$\sum_{n\in\mathcal{I}_{\lambda,0}}c_{n}u_{n}\geq\varepsilon\left\Vert \mathbf{u}_{\mathcal{I}_{\lambda,0}}\right\Vert _{2}+1$,
which by Lemma \ref{lem:exist eqv} is solvable if and only if $\tau_{0}>\varepsilon^{2}$;
a nonnegative solution is obtained from any solution $\mathbf{u}$ by taking $\left|\mathbf{u}\right|$.
This is case (i).

\emph{Case B: $\mu>0$ and $\mathbf{u}_{1:r}\neq\mathbf{0}$.} Put $k=\mu\varepsilon/\left\Vert \mathbf{u}_{1:r}\right\Vert _{2}>0$.
Then $\mathbf{s}=\mathbf{u}_{1:r}/\left\Vert \mathbf{u}_{1:r}\right\Vert _{2}$,
$\mu\varepsilon s_{n}=ku_{n}$, and \eqref{eq:KKT subgradient} becomes
$u_{n}=\mu c_{n}/\left(2\lambda_{n}+k\right)$ for $n\leq r$ and $u_{n}=\mu c_{n}/\left(2\lambda_{n}\right)$
for $n>r$ ($\lambda_{n}>0$ there, since $r<N$). Inserting $u_{n}$ into
the definition of $k$ and cancelling $\mu$ shows that the definition holds if
and only if $f\left(k\right)=\varepsilon^{2}$. Each term of $f$ is nondecreasing
in $k$, strictly increasing if $\lambda_{n}>0$ and $c_{n}>0$, and constant
($=c_{n}^{2}$) if $\lambda_{n}=0$; so $f$ is continuous and nondecreasing from
$f\left(0^{+}\right)=\tau_{0}$ to $f\left(+\infty\right)=\tau_{r}$, and strictly
increasing unless it is constant, in which case $\tau_{0}=\tau_{r}$. Therefore $f\left(k\right)=\varepsilon^{2}$
has a root $k>0$ if and only if $\tau_{0}<\varepsilon^{2}<\tau_{r}$, and the
root is unique. For such $k$, the activity of the constraint, $h\left(\mathbf{u}\right)=0$,
with $\left\Vert \mathbf{u}_{1:r}\right\Vert _{2}=\mu\varepsilon/k$ gives
\begin{equation}
\mu\left[\sum_{n=1}^{r}\frac{c_{n}^{2}}{2\lambda_{n}+k}-\frac{\varepsilon^{2}}{k}+\sum_{n=r+1}^{N}\frac{c_{n}^{2}}{2\lambda_{n}}\right]=1,
\end{equation}
and $\varepsilon^{2}/k=\sum_{n=1}^{r}c_{n}^{2}k/\left(2\lambda_{n}+k\right)^{2}$
(from $f\left(k\right)=\varepsilon^{2}$) turns the bracket into
$\sum_{n=1}^{r}\frac{2\lambda_{n}c_{n}^{2}}{\left(2\lambda_{n}+k\right)^{2}}+\sum_{n=r+1}^{N}\frac{c_{n}^{2}}{2\lambda_{n}}$,
which is positive: otherwise $\mathbf{c}$ would be supported on $\mathcal{I}_{\lambda,0}$
and $\tau_{0}=\tau_{r}$, contradicting $\tau_{0}<\tau_{r}$. This gives $\mu>0$ as stated, and
$\mathbf{u}_{1:r}\neq\mathbf{0}$ holds because $f\left(k\right)=\varepsilon^{2}>0$
forces some $c_{n}>0$, $n\leq r$. This is case (ii).

\emph{Case C: $\mathbf{u}_{1:r}=\mathbf{0}$.} If $\mu=0$, \eqref{eq:KKT subgradient}
gives $u_{n}=0$ for $n>r$ as well ($\lambda_{n}>0$ there), so $\mathbf{u}=\mathbf{0}$,
which is infeasible. Hence $\mu>0$ and $h\left(\mathbf{u}\right)=0$. For $n\leq r$,
\eqref{eq:KKT subgradient} reads $c_{n}=\varepsilon s_{n}$ with $\left\Vert \mathbf{s}\right\Vert _{2}\leq1$,
which is possible if and only if $\tau_{r}\leq\varepsilon^{2}$. For $n>r$,
$u_{n}=\mu c_{n}/\left(2\lambda_{n}\right)$ and $h\left(\mathbf{u}\right)=0$ reads
$\mu\sum_{n=r+1}^{N}c_{n}^{2}/\left(2\lambda_{n}\right)=1$, which is possible if
and only if $\mathcal{I}_{c,+}\neq\emptyset$, with $\mu=\left[\sum_{n=r+1}^{N}c_{n}^{2}/\left(2\lambda_{n}\right)\right]^{-1}$.
This is case (iii); it is void if $r=N$.

\emph{Step 2: existence.} The three regimes $\varepsilon^{2}<\tau_{0}$,
$\tau_{0}<\varepsilon^{2}<\tau_{r}$ and $\varepsilon^{2}\geq\tau_{r}$ are
disjoint, and since $\tau_{0}\leq\tau_{r}$ cases A, B and C can only occur in
the first, second and third, respectively. By Step 1, a solution of
\eqref{eq:opt prob} exists if and only if (i), (ii) or (iii) holds. In particular
none does when $\varepsilon^{2}=\tau_{0}$ (for $r=N$, $\mathcal{I}_{c,+}=\emptyset$ rules out C,
and for $r<N$, $\tau_{0}=0<\varepsilon^{2}$), nor when $\varepsilon^{2}\geq\tau_{r}$
and $\mathcal{I}_{c,+}=\emptyset$. The solution $\mathbf{w}^{\star}=\mathbf{T}\mathbf{v}^{\star}$
has the stated form by Lemma \ref{lem:phase align}.

\emph{Step 3: uniqueness.} By Lemma \ref{lem:unique criterion}, $p^{\star}=\mu/2$, where $\mu$ is the
multiplier of the KKT point found in Step 1. In case (i) this is case A, with $\mu=0$, so
$p^{\star}=0$; in cases (ii) and (iii) it is case B or C, with $\mu>0$, so $p^{\star}>0$.
The lemma then shows that the solution
of \eqref{eq:opt prob} is not unique in case (i) and is unique in cases (ii) and (iii).
That lemma is stated for problem \eqref{eq:opt prob} itself, so the uniqueness of
$\mathbf{w}^{\star}$ involves neither $\mathbf{T}$ nor $\mathbf{b}$: the solution is
the same whichever admissible $\mathbf{T}$ is used to compute it.
\end{IEEEproof}

\section{\label{sec:Proof-of-Lemma exist eqv}Proof of Lemma \ref{lem:exist eqv}}
\begin{IEEEproof}
Both directions rest on the Cauchy-Schwarz inequality, which bounds
$\mathbf{x}^{\T}\mathbf{y}$ by
$\left\Vert \mathbf{y}\right\Vert _{2}\left\Vert \mathbf{x}\right\Vert _{2}$ and
is tight precisely when $\mathbf{x}$ is aligned with $\mathbf{y}$. The condition
therefore amounts to asking whether $\mathbf{y}$ is long enough for the linear
term to outgrow $\varepsilon\left\Vert \mathbf{x}\right\Vert _{2}$.

For sufficiency, let $\left\Vert \mathbf{y}\right\Vert _{2}>\varepsilon$ and take
$\mathbf{x}$ along $\mathbf{y}$, namely
$\mathbf{x}=\rho\mathbf{y}/\left\Vert \mathbf{y}\right\Vert _{2}$ with
$\rho=1/\left(\left\Vert \mathbf{y}\right\Vert _{2}-\varepsilon\right)>0$, which
is well defined by the assumption. Then
$\left\Vert \mathbf{x}\right\Vert _{2}=\rho$ and
\begin{equation}
\mathbf{x}^{\T}\mathbf{y}=\rho\left\Vert \mathbf{y}\right\Vert _{2}=\frac{\left\Vert \mathbf{y}\right\Vert _{2}}{\left\Vert \mathbf{y}\right\Vert _{2}-\varepsilon}=1+\varepsilon\rho=1+\varepsilon\left\Vert \mathbf{x}\right\Vert _{2},
\end{equation}
so this $\mathbf{x}$ satisfies the constraint, with equality.

For necessity, argue by contraposition. If
$\left\Vert \mathbf{y}\right\Vert _{2}\leq\varepsilon$, then for every
$\mathbf{x}$
\begin{equation}
\mathbf{x}^{\T}\mathbf{y}\leq\left\Vert \mathbf{y}\right\Vert _{2}\left\Vert \mathbf{x}\right\Vert _{2}\leq\varepsilon\left\Vert \mathbf{x}\right\Vert _{2}<\varepsilon\left\Vert \mathbf{x}\right\Vert _{2}+1,
\end{equation}
so no $\mathbf{x}$ satisfies the constraint and the set is empty. Note that the
gap closed by the alignment in the first part is exactly the one Cauchy-Schwarz
leaves open in the second. 
\end{IEEEproof}
\end{appendices}
\end{document}